\documentclass[a4paper]{article}
\usepackage{fullpage}
\usepackage{microtype}
\usepackage{colonequals}
\usepackage{hyperref}
\usepackage{graphicx}
\usepackage{amsthm,amsmath,amsfonts,amssymb,stmaryrd,mathrsfs,mathdots,amsbsy,float,cmll,xcolor,mathpartir,lmodern,centernot}
\usepackage{bm}
\usepackage{mathtools,mathpartir,cmll}
\usepackage{tikz}
\usepackage{tikz-cd}
\newsavebox{\pullback}
\sbox\pullback{%
\begin{tikzpicture}%
\draw (0,0) -- (1ex,0ex);%
\draw (1ex,0ex) -- (1ex,1ex);%
\end{tikzpicture}}
\usetikzlibrary{tikzmark}
\usepackage[small]{diagrams}
\usepackage{float}
\usetikzlibrary{positioning,calc}
\usepackage{bussproof}

\makeatletter
\DeclareRobustCommand{\notsquare}{\mathord{\mathpalette\generic@not\square}}
\newcommand{\generic@not}[2]{%
  \sbox\z@{$\m@th#1/$}%
  \sbox\tw@{$\m@th#1#2$}%
  \sbox\z@{\raisebox{\dimexpr(\ht\tw@-\dp\tw@-\ht\z@+\dp\z@)/2\relax}{$\m@th#1/$}}%
  \vphantom{\usebox{\z@}}%
  \ooalign{\hidewidth\usebox{\z@}\hidewidth\cr$\m@th#1#2$\cr}%
}
\makeatother

\theoremstyle{definition}
\newtheorem{definition}{Definition}[section]
\newtheorem{example}[definition]{Example}
\theoremstyle{theorem}
\newtheorem{theorem}[definition]{Theorem}
\newtheorem*{theorem*}{Theorem}
\newtheorem{proposition}[definition]{Proposition}

\newtheorem{corollary}[definition]{Corollary}
\newtheorem*{corollary*}{Corollary}
\theoremstyle{remark}
\newtheorem*{remark}{Remark}
\newtheorem*{notation}{Notation}
\newtheorem*{convention}{Convention}

\newarrow{Corresponds}<--->
\newarrow {Dashto} {}{dash}{}{dash}>
\newarrow{Equals}=====

\DeclareFontFamily{U}{MnSymbolC}{}
\DeclareFontShape{U}{MnSymbolC}{m}{n}{
    <-6>  MnSymbolC5
   <6-7>  MnSymbolC6
   <7-8>  MnSymbolC7
   <8-9>  MnSymbolC8
   <9-10> MnSymbolC9
  <10-12> MnSymbolC10
  <12->   MnSymbolC12}{}
\DeclareSymbolFont{MnSyC}{U}{MnSymbolC}{m}{n}
\DeclareMathSymbol{\dotminus}{\mathbin}{MnSyC}{24}

\begin{document}
\title{\Huge \bfseries Game semantics of universes}
\author{\Large \bfseries Norihiro Yamada \\ \\
{\tt yamad041@umn.edu} \\
School of Mathematics \\
University of Minnesota
}

\maketitle

\begin{abstract}
This work extends the present author's computational game semantics of Martin-L\"{o}f type theory to the cumulative hierarchy of \emph{universes}. 
This extension completes game semantics of all standard types of Martin-L\"{o}f type theory for the first time in the 30 years history of modern game semantics.\footnote{More extensional, domain and realisability semantics of universes has been established though \cite{palmgren1993information,streicher2012semantics,blot2018extensional}.}
As a result, the powerful combinatorial reasoning of game semantics becomes available for the study of universes and types generated by them. 
A main challenge in achieving game semantics of universes comes from a \emph{conflict} between identity types and universes: Naive game semantics of the encoding of an identity type by a universe induces a decision procedure on the equality between functions, a contradiction to a well-known fact in recursion theory.
We overcome this problem by novel games for universes that encode games for identity types \emph{without deciding the equality}.
\end{abstract}

\maketitle


\tableofcontents

\section{Introduction}
\label{Introduction}
For this introduction, we assume that the reader is familiar with the syntax of dependent type theories and universes \cite{hofmann1997syntax}, but not with game semantics \cite{abramsky1997semantics,hyland1997game}.

\subsection{Martin-L\"{o}f type theory and the meaning explanation}
On the one hand, \emph{formal systems} \cite{shoenfield1967mathematical} are a class of symbolic or \emph{syntactic} formalisations of mathematics, and \emph{constructive mathematics} \cite{troelstra1988constructivism} is a family of computational or \emph{constructive} schools in mathematics.
On the other hand, \emph{semantics} of a formal system is an assignment of syntax-free objects to syntactic objects of the formal system, where the former serves as the `meaning' or \emph{interpretation} of the latter.

\emph{Martin-L\"{o}f type theory (MLTT)} \cite{martin1975intuitionistic,martin1984intuitionistic,martin1998intuitionistic} is a prominent formal system for constructive mathematics, and it is comparable to axiomatic set theory \cite{zermelo1908untersuchungen,fraenkel1922grundlagen} for classical mathematics.
The fundamental idea of MLTT is to regard (mathematical) objects and proofs in constructive mathematics uniformly as \emph{computations} in an informal sense, and MLTT is a syntactic formalisation of this beautiful idea \cite{martin1982constructive}.
Hence, objects and proofs in MLTT are unified into \emph{terms}, where formulas are called \emph{types}.
This standard yet informal semantics of MLTT is called the \emph{meaning explanation} \cite[\S 5]{dybjer2016intuitionistic}.

Nevertheless, by its informal nature, the meaning explanation cannot serve as a mathematically firm ground to analyse, justify or develop MLTT.
Besides, MLTT is an intricate formal system that inevitably contains superficial syntactic details, which makes it difficult to study the meta-theory of MLTT.

\subsection{Game semantics of Martin-L\"{o}f type theory}
This problem calls for \emph{mathematical} semantics \cite{gunter1992semantics} of MLTT that faithfully formalises the meaning explanation, abstracting the inessential syntactic details, and advances the meta-theoretic study of MLTT.
Motivated in this way, the present author has established \emph{game semantics of MLTT} \cite{yamada2022game}.

\emph{Game semantics} \cite{abramsky1997semantics,hyland1997game} is a class of mathematical semantics that interprets types by \emph{games} between \emph{Player} (or a mathematician) and \emph{Opponent} (or an oracle), and terms by \emph{strategies} for Player on how to play on games. 
Games are a class of rooted directed forests, and strategies are algorithms for Player on how to walk on (or play) games alternately with Opponent in such a way that it is Player's \emph{win}.

We write a walk or \emph{play} in a game by a potentially infinite sequence of finite sequences 
\begin{equation*}
\boldsymbol{\epsilon}, m_1, m_1m_2, m_1m_2m_3, \dots,
\end{equation*}
where $\boldsymbol{\epsilon}$ is the empty sequence, each element or \emph{move} $m_i$ is a vertex of the game, and each sequence or \emph{position} $m_1m_2 \dots m_n$ is a finite path from the root in the game.
By convention, the first move $m_1$ is always made by Opponent, and then Player and Opponent alternately make moves.
Thus, the moves $m_{2i+1}$ are made by Opponent, and the other ones $m_{2i}$ by Player ($i \in \mathbb{N}$). 
Because a strategy describes the next move to be made by Player, if any, we describe its computational step by the partial function
\begin{equation*}
m_1 \mapsto m_2, m_1m_2m_3 \mapsto m_4, \dots.
\end{equation*}

The game semantics of MLTT formalises the meaning explanation \emph{syntax-independently} and \emph{intuitively} by interpreting terms as strategies or \emph{interactive computations} between Player and Opponent.
In addition, the game semantics turns out to be a highly effective tool for the meta-theoretic study of MLTT; e.g., it verifies the \emph{independence of Markov's principle} \cite{yamada2022game}, which is not possible by most other mathematical semantics of MLTT such as Hyland's \emph{effective topos} \cite{hyland1982effective}.
The point is that game semantics is unique in its interpretation of terms by strategies or \emph{intensional processes}, while other mathematical semantics interprets terms by extensional objects such as functions.
Because terms are also intensional objects, computing in a step-by-step fashion, game semantics achieves a very tight correspondence between terms and strategies, which makes itself an exceptionally powerful tool for the study of formal systems.

\subsection{Examples of games and strategies}
\label{ExamplesOfGamesAndStrategies}
In the following, let us see some simple examples of games and strategies as a preparation for \S\ref{HowToEncodeGamesByStrategies}.
For instance, the game $N$ of natural numbers is the rooted tree (which is infinite in width)
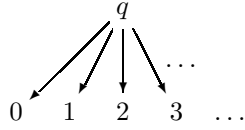
\begin{diagram}
& & q & & \\
& \ldTo(2, 2) \ldTo(1, 2) & \dTo & \rdTo(1, 2) \ \dots & \\
0 & 1 & 2 & 3 & \dots
\end{diagram}
in which a play starts with Opponent's move or question $q$ (`What is your number?') and ends with Player's move or answer $n \in \mathbb{N}$ (`My number is $n$!'). 
This natural number game $N$ is not very different from the set $\mathbb{N}$ of all natural numbers, and there is a much more intensional game for natural numbers \cite{yamada2019game}. 
However, the game $N$ is simpler and suffices for our purpose. 
A strategy $\underline{7}$ on $N$, written $\underline{7} : N$, corresponding to the number $7 \in \mathbb{N}$ for instance, is the map $q \mapsto 10$.
In the rest of this introduction, we describe games by listing their positions. 
For example, the set of all positions of $N$ is $\{ \boldsymbol{\epsilon}, q \} \cup \{ \, qn \mid n \in \mathbb{N} \, \}$.

There is a binary construction $\&$ on games, called \emph{product}, which forms binary product in the category of games and strategies.
The product $A \mathbin{\&} B$ of games $A$ and $B$ are simply the disjoint union of $A$ and $B$.
In other words, a position of $A \mathbin{\&} B$ is either a position of $A$ or $B$.
For instance, a maximal position of the product $N \mathbin{\&} N$ of the game $N$ and itself is either of the following forms\footnote{The diagrams are only to make it explicit which component game each move belongs to; the two positions are just finite sequences $q_{[0]} n_{[0]}$ and $q_{[1]} m_{[1]}$.}:
\begin{mathpar}
\begin{tabular}{ccccccccc}
$N_{[0]}$ & $\&$ & $N_{[1]}$ &&&& $N_{[0]}$ & $\&$ & $N_{[1]}$ \\ \cline{1-3} \cline{7-9}
$q_{[0]}$&&&&&&&& $q_{[1]}$ \\
$n_{[0]}$&&&&&&&& $m_{[1]}$
\end{tabular}
\end{mathpar}
where $n, m \in \mathbb{N}$, and the subscripts $(\_)_{[i]}$ ($i = 0, 1$) are arbitrary \emph{tags} to distinguish the two copies of $N$. 
We often omit the tags $(\_)_{[i]}$ when it does not bring confusion.
We write $\langle \underline{n}, \underline{m} \rangle$ for the strategy on $N \mathbin{\&} N$ that plays as in the above diagrams, which forms the \emph{pairing} of the strategies $\underline{n}, \underline{m} : N$. 

Another central construction $\multimap$, called \emph{linear implication}, captures the notion of \emph{linear functions}, i.e., functions that consume exactly one input to produce an output. 
A position of the linear implication $A \multimap B$ between $A$ and $B$ is an interleaving mixture of a position of $A$ and a position of $B$ such that
\begin{enumerate}

\item The first element of the position must be a move of $B$; 

\item A change of the $AB$-parity in the position must be made by Player.

\end{enumerate}
For example, a typical position of the linear implication $N \multimap N$ is
\begin{mathpar}
\begin{tabular}{ccc}
$N_{[0]}$ & $\multimap$ & $N_{[1]}$ \\ \hline 
&&$q_{[1]} $ \\
$q_{[0]}$ && \\
$n_{[0]}$&& \\
&&$m_{[1]}$
\end{tabular}
\end{mathpar}
where $n, m \in \mathbb{N}$, which can be read as follows:
\begin{enumerate}
\item Opponent's question $q_{[1]}$ for an output (`What is your output?');
\item Player's question $q_{[0]}$ for an input (`Wait, what is your input?');
\item Opponent's answer, say, $n_{[0]}$, to $q_{[0]}$ (`OK, here is an input $n$.');
\item Player's answer, say, $m_{[1]}$, to $q_{[1]}$ (`Alright, the output is then $m$.').
\end{enumerate}
This play corresponds to any linear function $\mathbb{N} \rightarrow \mathbb{N}$ that maps $n \mapsto m$.
The strategy $\mathrm{succ}$ on $N \multimap N$ for the successor function is the map $q_{[1]} \mapsto q_{[0]}, q_{[1]}q_{[0]}n_{[0]} \mapsto n+1_{[1]}$, or diagrammatically 
\begin{mathpar}
\begin{tabular}{ccc}
$N_{[0]}$ & $\stackrel{\mathrm{succ}}{\multimap}$ & $N_{[1]}$ \\ \hline 
&&$q_{[1]}$ \\
$q_{[0]}$&& \\
$n_{[0]}$&& \\
&&$n+1_{[1]}$
\end{tabular}
\end{mathpar}

Let us remark here that the following play, which corresponds to a \emph{constant} linear function that maps $x \mapsto m$ for all $x \in \mathbb{N}$, is also possible: $q_{[1]} \mapsto m_{[1]}$.
Thus, strictly speaking, $A \multimap B$ is the game of \emph{affine functions} from $A$ to $B$, but we follow the standard convention to call $\multimap$ linear implication. 

However, the linear implication $N \mathbin{\&} N \multimap N$ cannot accommodate strategies that compute binary functions such as addition because maximal positions of this game are either of the following forms:
\begin{mathpar}
\begin{tabular}{ccccc}
$N$ & $\&$ & $N$ & $\multimap$ & $N$ \\ \hline 
&&&&$q $ \\
$q$ && \\
$n$&& \\
&&&&$m$
\end{tabular}
\and
\begin{tabular}{ccccc}
$N$ & $\&$ & $N$ & $\multimap$ & $N$ \\ \hline 
&&&&$q $ \\
&&$q$ && \\
&&$n$&& \\
&&&&$m$
\end{tabular}
\and
\begin{tabular}{ccccc}
$N$ & $\&$ & $N$ & $\multimap$ & $N$ \\ \hline 
&&&&$q $ \\
&&&&$m$
\end{tabular}
\end{mathpar}

The unary construction $\oc$ on games, called \emph{exponential}, addresses this problem by defining the desired game $A \Rightarrow B$ for ordinary (not necessarily linear) functions from $A$ to $B$ by $A \Rightarrow B \colonequals \oc A \multimap B$.
This idea comes from \emph{linear logic} \cite{girard1987linear}.
A position of the exponential $\oc A$ is an interleaving mixture of a finite number of positions of $A$ such that a switch between different copies of positions of $A$ inside $\oc A$ must be made by Opponent. 
For instance, the exponential $\oc (N \mathbin{\&} N)$ accommodates the positions
\begin{mathpar}
\begin{tabular}{ccccc}
$\oc (N$ & $\&$ & $N)$ \\ \hline 
$q $ \\
$n$ \\
&& $q$ \\
&& $m$
\end{tabular}
\and
\begin{tabular}{ccccc}
$\oc (N$ & $\&$ & $N)$ \\ \hline 
&&$q $ \\
&&$n$ \\
$q$ \\
$m$
\end{tabular}
\end{mathpar}
so that there is are strategies 
\begin{mathpar}
\begin{tabular}{ccccc}
$N$ & $\&$ & $N$ & $\stackrel{\mathrm{add}}{\Rightarrow}$ & $N$ \\ \hline 
&&&&$q $ \\
$q$ && \\
$n$&& \\
&&$q$ && \\
&& $m$ && \\
&&$$&& \\
&&&&$n+m$
\end{tabular}
\and
\begin{tabular}{ccccc}
$N$ & $\&$ & $N$ & $\stackrel{\mathrm{add}'}{\Rightarrow}$ & $N$ \\ \hline 
&&&&$q $ \\
&&$q$ && \\
&&$n$&& \\
$q$ && \\
$m$ && \\
&&$$&& \\
&&&&$n+m$
\end{tabular}
\end{mathpar}
both of which compute addition of natural numbers.
These strategies both implement addition, but their algorithms are slightly different, which illustrates the \emph{intensional} nature of game semantics. 

At this point, let us consider the game $(N \Rightarrow N) \Rightarrow N$ of \emph{higher-order} functions, which is higher-order because the domain $N \Rightarrow N$ is the game of functions. 
Note that the domain is the exponential $\oc (N \Rightarrow N)$, so a strategy $\phi$ on the game $(N \Rightarrow N) \Rightarrow N$ may interact with an input strategy $f$ on $\oc (N \Rightarrow N)$ given by Opponent \emph{any finite number of times}. 
Each interaction between $\phi$ and $f$ reveals an input-output pair of $f$, but this process will never collect the complete information about $f$ because there are infinitely many input-output pairs of $f$.
For instance, the strategy $\mathrm{pazo} : (N \Rightarrow N) \Rightarrow N$ that computes the sum $f(0)+f(1)$ for a given function $f : \mathbb{N} \Rightarrow \mathbb{N}$ plays by
\begin{mathpar}
\begin{tabular}{ccccc}
$\oc (\oc N_{[0]}$ & $\multimap$ & $N_{[1]})$ & $\stackrel{\mathrm{pazo}}{\multimap}$ & $N_{[2]}$ \\ \hline 
&&&& $q_{[2]}$ \\
&& $q_{[1]}$ && \\
$q_{[0]}$ && && \\
$0_{[0]}$ && && \\
&& $n_{[1]}$&& \\
&& $q_{[1]}$ && \\
$q_{[0]}$ && && \\
$1_{[0]}$&& && \\
&& $m_{[1]}$ && \\
&&&& $n+m_{[2]}$
\end{tabular}
\end{mathpar}
This play can be read as follows:
\begin{enumerate}
\item Opponent's question $q_{[2]}$ for an output (`What is your output?');
\item Player's question $q_{[1]}$ for an input function (`Wait, your first output please!');
\item Opponent's question $q_{[0]}$ for an input (`What is your first input then?');
\item Player's answer, say, $0_{[0]}$, to the question $q_{[0]}$ (`Here is my first input $0$.');
\item Opponent's answer, say, $n_{[1]}$, to the question $q_{[1]}$ (`OK, then here is my first output $n$.');
\item Player's question $q_{[1]}$ for an input function (`Your second output please!');
\item Opponent's question $q_{[0]}$ for an input (`What is your second input then?');
\item Player's answer, say, $1_{[0]}$, to the question $q_{[0]}$ (`Here is my second input $1$.');
\item Opponent's answer, say, $m_{[1]}$, to the question $q_{[1]}$ (`OK, then here is my second output $m$.');
\item Player's answer, say, $n+m_{[2]}$, to the question $q_{[2]}$ (`Alright, my output is then $n+m$.').
\end{enumerate}
In this play, the strategy $\mathrm{pazo} $ has only revealed the two input-output pairs $(0, n)$ and $(1, m)$ of $f$.

Finally, let us recall the \emph{composition} $\psi \bullet \phi : A \Rightarrow C$ of strategies $\phi : A \Rightarrow B$ and $\psi : A \Rightarrow C$.
For an illustration, consider the strategies $\mathrm{succ}, \mathrm{double} : N \Rightarrow N$ (n.b., $\mathrm{succ}$ this time is not on $N \multimap N$):
\begin{mathpar}
\begin{tabular}{ccc}
$N_{[0]}$ & $\stackrel{\mathrm{succ}}{\Rightarrow}$ & $N_{[1]}$  \\ \hline
&&$q_{[1]}$  \\
$q_{[0]}$&& \\
$m_{[0]}$ && \\
&&$m+1_{[1]}$
\end{tabular}
\and
\begin{tabular}{ccc}
$N_{[2]}$ & $\stackrel{\mathrm{double}}{\Rightarrow}$ & $N_{[3]}$  \\ \hline
&&$q_{[3]}$  \\
$q_{[2]}$&& \\
$n_{[3]}$ && \\
&&$2n_{[2]}$
\end{tabular}
\end{mathpar}
The composition $\mathrm{double} \bullet \mathrm{succ} : N \Rightarrow N$ is calculated as follows. 
First, we have to define the \emph{promotion} $\mathrm{succ}^\dagger : \oc N_{[0]} \multimap \oc N_{[1]}$ of $\mathrm{succ}$, which computes just as $\mathrm{succ} : \oc N_{[0]} \multimap N_{[1]}$ for each position of $\oc N_{[0]} \multimap N_{[1]}$ occurring inside $\oc N_{[0]} \multimap \oc N_{[1]}$.
A typical position played by the promotion therefore looks like 
\begin{mathpar}
\begin{tabular}{ccc}
$\oc N_{[0]}$ & $\stackrel{\mathrm{succ}^\dagger}{\multimap}$ & $\oc N_{[1]}$  \\ \hline
&&$q_{[1]}$  \\
$q_{[0]}$&& \\
$m_{[0]}$ && \\
&&$m+1_{[1]}$ \\
&&$q_{[1]}$  \\
$q_{[0]}$&& \\
$m'_{[0]}$ && \\
&&$m'+1_{[1]}$ \\
&&$q_{[1]}$  \\
$q_{[0]}$&& \\
$m''_{[0]}$ && \\
&&$m''+1_{[1]}$
\end{tabular}
\end{mathpar}

Next, we \emph{synchronise} $\mathrm{succ}^\dagger$ and $\mathrm{double}$ via the codomain $\oc N_{[1]}$ of $\mathrm{succ}^\dagger$ and the domain $\oc N_{[2]}$ of $\mathrm{double}$, for which Player also plays the role of Opponent in $\oc N_{[1]}$ and $\oc N_{[2]}$ by copying her last moves, resulting in
\begin{mathpar}
\begin{tabular}{ccccccc}
$\oc N_{[0]}$ & $\stackrel{\mathrm{succ}^\dagger}{\multimap}$ & $\oc N_{[1]}$ && $\oc N_{[2]}$ & $\stackrel{\mathrm{double}}{\multimap}$ & $N_{[3]}$ \\ \hline
&&&&&& $q_{[3]}$ \\
&&&& \fbox{$q_{[2]}$} && \\
&& \fbox{$q_{[1]}$} &&&& \\
$q_{[0]}$ &&&&&& \\
$n_{[0]}$ &&&&&& \\
&& \fbox{$n+1_{[1]}$} &&&& \\
&&&&\fbox{$n+1_{[2]}$} && \\
&&&&&&$2 \cdot (n+1)_{[3]}$
\end{tabular}
\end{mathpar}
where moves made for the synchronisation are marked by the square boxes just for clarity. 
Importantly, it is assumed that Opponent plays on the \emph{external game} $N_{[0]} \Rightarrow N_{[3]}$, seeing only moves of $\oc N_{[0]}$ or $N_{[3]}$.

The resulting play is to be read as follows:
\begin{enumerate}

\item Opponent's question $q_{[3]}$ for an output in $\oc N_{[0]} \multimap N_{[3]}$ (`What is your output?');

\item Player's question \fbox{$q_{[2]}$} by $\mathrm{double}$ for an input in $\oc N_{[2]} \multimap N_{[3]}$ (`Wait, what is your input?');

\item \fbox{$q_{[2]}$} in turn triggers the question \fbox{$q_{[1]}$} for an output in $\oc N_{[0]} \multimap \oc N_{[1]}$ (`What is your output?');

\item Player's question $q_{[0]}$ by $\mathrm{succ}^\dagger$ for an input in $\oc N_{[0]} \multimap \oc N_{[1]}$ (`Wait, what is your input?');

\item Opponent's answer, say, $n_{[0]}$, to $q_{[0]}$ in $\oc N_{[0]} \multimap \oc N_{[3]}$ (`Here is an input $n$.');

\item Player's answer \fbox{$n+1_{[1]}$} to \fbox{$q_{[1]}$} by $\mathrm{succ}^\dagger$ in $\oc N_{[0]} \multimap \oc N_{[1]}$ (`The output is then $n+1$.');

\item \fbox{$n+1_{[1]}$} in turn triggers the answer \fbox{$n+1_{[2]}$} to \fbox{$q_{[2]}$} in $\oc N_{[2]} \multimap N_{[3]}$ (`Here is the input $n+1$.');

\item Player's answer $2 \cdot (n+1)_{[3]}$ to $q_{[3]}$ by $\mathrm{double}$ in $\oc N_{[0]} \multimap N_{[3]}$ (`The output is $2 \cdot (n+1)$!').

\end{enumerate}

Finally, we hide or delete all moves with the square boxes from the play, resulting in the strategy $\mathrm{double} \bullet \mathrm{succ} : N \Rightarrow N$ for the function $n \mapsto 2 \cdot (n + 1)$ as expected:
\begin{mathpar}
\begin{tabular}{ccc}
$N_{[0]}$ & $\stackrel{\mathrm{double} \bullet \mathrm{succ}}{\Rightarrow}$ & $N_{[3]}$ \\ \hline
&& $q_{[3]}$ \\
$q_{[0]}$&& \\
$n_{[0]}$&& \\
&&$2 \cdot (n + 1)_{[3]}$
\end{tabular}
\end{mathpar}
The category of games and strategies has games as objects, and strategies $\phi : A \Rightarrow B$ as morphisms $A \rightarrow B$, and the composition of strategies just sketched forms the categorical composition.

Moreover, one can compose strategies $\alpha : A$ and $\phi : A \Rightarrow B$ in the same vein, obtaining the composition $\phi \bullet \alpha : B$.
For instance, we have the composition $\mathrm{double} \bullet \underline{n} = \underline{2n}$ for all $n \in \mathbb{N}$.
Alternatively, recall the \emph{terminal game} $T$, which has no move.
Hence, we have the isomorphism $A \cong T \Rightarrow A$, and we do not distinguish strategies on $A$ and $T \Rightarrow A$ since they are essentially the same.
As a result, the composition $\phi \bullet \alpha : B$ can be recasted as the ordinary composition $\phi \bullet \alpha : T \Rightarrow B$ of $\alpha : T \Rightarrow A$ and $\phi : A \Rightarrow B$.

We have seen that strategies interact with each other in a \emph{step-by-step, finitary} fashion.
This unique, \emph{intensional computation} distinguishes game semantics from other mathematical semantics.

\subsection{Martin-L\"{o}f's universes}
One can extend MLTT by a `types of (smaller) types' or \emph{universe} introduced by Martin-L\"{o}f \cite{martin1975intuitionistic}. 
The universe enables MLTT to expand its realm of constructive mathematics significantly.
For instance, the elimination rule of the natural number (N-) type with respect to the universe generates \emph{infinitely} indexed dependent types such as the type of finite lists of natural numbers by mathematical induction.

Besides, the power of the universe is greatly increased when it is combined with Martin-L\"{o}f's \emph{well-founded tree (W-) types} \cite{martin1982constructive}.
For instance, MLTT together with the universe and W-types interprets Aczel's constructive set theory \cite{aczel1986type}.
Moreover, the combination of the universe and W-types offers MLTT a high proof-theoretic strength among constructive formal systems \cite{setzer1993proof,griffor1994strength}.

\subsection{The problem: how to encode games for identity types by strategies}
\label{HowToEncodeGamesByStrategies}
For these significant roles of the universe in MLTT and constructive mathematics, it is a natural aim to extend game semantics to the universe so that its powerful combinatorial reasoning becomes available for the study of the universe and types generated by the universe. 
However, it is a challenge to achieve game semantics of the universe, and it has not been established in the 30 years history of game semantics.\footnote{Blot and Laird \cite{blot2018extensional} interpret universes, but this interpretation is by domain theory, not by game semantics.}

Specifically, the challenge is \emph{how to encode games for identity (Id-) types by strategies}. 
To see this point, recall first that the game semantics \cite{yamada2022game} interprets each dependent type $\mathsf{\Gamma \vdash A \ type}$ roughly by a family $A = (A(\gamma))_{\gamma : \Gamma}$ of games $A(\gamma)$ indexed by strategies $\gamma$ on the game $\Gamma$ that interprets the context $\mathsf{\Gamma}$. 
Also, recall that the introduction rule of each universe $\mathsf{\Gamma \vdash U \ type}$ \emph{encodes} the dependent type $\mathsf{\Gamma \vdash A \ type}$ by a term $\mathsf{\Gamma \vdash En(A) : U}$ in such a way that the computation rule $\mathsf{\Gamma \vdash El(En(A)) = A \ type}$ holds, where the dependent type $\mathsf{x : U \vdash El(x) \ type}$ embodies the elimination rule of the universe by the substitution $\mathsf{\Gamma \vdash u : U} \mapsto \mathsf{\Gamma \vdash El(u) \ type}$.
Note that each universe is a \emph{constant} dependent type, and therefore the game semantics \cite{yamada2022game} should interpret it by a \emph{constant} family of games, which is in turn identified by a single game in the evident way.
Hence, we have to define not only a game $\mathcal{U}$ that interprets the universe $\mathsf{U}$ but also the corresponding encoding of the family $A$ of games by a strategy $\mathrm{En}(A)$ on the function game $\Gamma \Rightarrow \mathcal{U}$ from $\Gamma$ to $\mathcal{U}$, which interprets the introduction rule, and a family $\mathrm{El} = (\mathrm{El}(\mu))_{\mu : \mathcal{U}}$ of games $\mathrm{El}(\mu)$, which interprets the elimination rule, that satisfies $\mathrm{El}(\mathrm{En}(A) \bullet \gamma) = A(\gamma)$ for all $\gamma : \Gamma$, which interprets the computation rule.
Recall that a strategy on the game $\Gamma \Rightarrow \mathcal{U}$ is a certain kind of an \emph{algorithm} that outputs a strategy on the codomain $\mathcal{U}$ from a given input strategy on the domain $\Gamma$.

Now, let us take $\mathrm{A}$ to be the Id-type $\mathsf{f : N \Rightarrow N, g : N \Rightarrow N \vdash Id_{N \Rightarrow N}(f, g) \ type}$ on the function type $\mathsf{N \Rightarrow N}$ from N-type $\mathsf{N}$ to itself. 
Then, the game semantics \cite{yamada2022game} has to interpret the encoding term $\mathsf{f : N \Rightarrow N, g : N \Rightarrow N \vdash \mathrm{En}(Id_{N \Rightarrow N}(f, g)) : U}$ for this Id-type by a strategy $\mathrm{En}(\mathrm{Id}_{N \Rightarrow N}) : (N \Rightarrow N) \mathbin{\&} (N \Rightarrow N) \Rightarrow \mathcal{U}$ that satisfies $\mathrm{El}(\mathrm{En}(\mathrm{Id}_{N \Rightarrow N}) \bullet \langle f, g \rangle) = \mathrm{Id}_{N \Rightarrow N}(\langle f, g \rangle)$ for all $f, g : N \Rightarrow N$, where $\mathrm{Id}_{N \Rightarrow N}$ is the family of games that interprets the Id-type \cite{yamada2022game}.
Crucially, the game $\mathrm{Id}_{N \Rightarrow N}(\langle f, g \rangle)$ depends on the equation $f = g$. 
Thus, the composition $\mathrm{En}(\mathrm{Id}_{N \Rightarrow N}) \bullet \langle f, g \rangle$ \emph{must vary over the cases $f = g$ and $f \neq g$}.

Accordingly, the strategy $\mathrm{En}(\mathrm{Id}_{N \Rightarrow N})$ seems to be an algorithm that decides whether the equation $f = g$ holds for all $f, g : N \Rightarrow N$, a contradiction to a well-known fact in recursion theory \cite{rogers1967theory}.
This corresponds, in game semantics, to that the strategy $\mathrm{En}(\mathrm{Id}_{N \Rightarrow N})$ can learn about only \emph{finite} input-output pairs of $f$ and $g$, so it cannot decide if the equation $f = g$ holds, as illustrated by the diagram  

\begin{mathpar}
\begin{tabular}{ccccccccc}
$(N$ & $\stackrel{f}{\Rightarrow}$ & $N)$ & $\&$ & $(N$ & $\stackrel{g}{\Rightarrow}$ & $N)$ & $\stackrel{\mathrm{En}(\mathrm{Id}_{N \Rightarrow N})}{\Rightarrow}$ & $\mathcal{U}$ \\ \hline
&&&&&&&&$u_1$ \\
&&$q$&&&&&& \\
$q$&&&&&&&& \\
$n$&&&&&&&& \\
&&$f(n)$&&&&&& \\
&&&&&&$q$&& \\
&&&&$q$&&&& \\
&&&&$n'$&&&& \\
&&&&&&$f(n')$&& \\
&&&&&&&$\vdots$& \\
&&&&&&&$\wn$&
\end{tabular}
\end{mathpar}

Let us see more concretely that the following naive method fails due to the problem just sketched.
Let us assign a natural number $\sharp(A)$ to the game $A$ that interprets each type $\mathsf{\Gamma \vdash A \ type}$ along the inductive construction of $\mathrm{A}$, and define a game $\mathcal{U}$ in such a way that maximal positions in $\mathcal{U}$ are of the form $q . \sharp(A)$. 
\begin{mathpar}
\begin{tabular}{c}
$\mathcal{U}$ \\ \hline
$q$ \\
$\sharp(A)$
\end{tabular}
\end{mathpar}
Intuitively, the initial element $q$ is Opponent's question `What is your game?', and the second one $\sharp(\mathrm{A})$ is Player's answer `My game is $A$!'.
Further, let $\mathrm{En}(A) : \Gamma \Rightarrow \mathcal{U}$ be the strategy that encodes the family $A$ of games by playing $q \mapsto \sharp(A)$ without ever computing on the domain $\Gamma$.
For this game $\mathcal{U}$, the strategy $\mathrm{En}(\mathrm{Id}_{N \Rightarrow N})$ would decide if $f = g$ (even without interacting with $f$ or $g$), which is clearly impossible.
\begin{mathpar}
\begin{tabular}{ccc}
$\Gamma$ & $\stackrel{\mathrm{En}(A)}{\not\Rightarrow}$ & $\mathcal{U}$ \\ \hline
&&$q$ \\
&&$\sharp(\mathrm{Id}_{N \Rightarrow N}(\langle f, g \rangle))$
\end{tabular}
\end{mathpar}

\begin{remark}
This naive method does not exploit any intrinsic feature of game semantics, but it actually works for encoding all standard dependent types except Id-types.
Hence, one may say that our main contribution is game semantics of the universe that subsumes the encoding of Id-types.
\end{remark}

\subsection{Our solution: encoding without deciding}
\label{Solution}
A key observation behind our solution to the problem described in \S\ref{HowToEncodeGamesByStrategies} is that
\begin{quote}
The game semantics \cite{yamada2022game} allows the decoding function $\mathrm{El}$ to be \emph{uncomputable} without sacrificing the algorithmic nature of strategies.\footnote{This is because the game semantics of Pi-types (Definition~\ref{DefLinearPiSpaces}) reveals the type dependency \emph{only gradually} so that it is not necessary to compute the value of the function $\mathrm{El}$ in one go. We shall come back to this point in Example~\ref{ExCrucialExample}.} 
In particular, the strategy $\mathrm{En}(\mathrm{Id}_{N \Rightarrow N}) \bullet \langle f, g \rangle$ \emph{does not have to decide} the equality $f = g$; it only has to encode the game $\mathrm{Id}_{N \Rightarrow N}(\langle f, g \rangle)$. 
\end{quote}

This leads us to the following solution.
Let $\sharp(1), \sharp(0), \sharp(N), \sharp(\Pi), \sharp(\Sigma), \sharp(\mathrm{Id}) \in \mathbb{N}$ be arbitrarily fixed pairwise distinct natural numbers.
We then define the game $\mathcal{U}$ in such a way that 
\begin{quote}
The strategy $\mathrm{En}(\mathrm{Id}_{N \Rightarrow N}) \bullet \langle f, g \rangle : \mathcal{U}$ plays first by computing $q \mapsto \sharp(\mathrm{Id})$ (indicating that it encodes an Id-type) and then, depending on the next move by Opponent, by playing as the strategy $\mathrm{En}(N \Rightarrow N)$ (indicating that the encoded Id-type is on the type $\mathsf{N \Rightarrow N}$) or by merely \emph{copy-catting $f$ and $g$} given by Opponent in the step-by-step fashion (indicating that the encoded Id-type is between $f$ and $g$) \emph{without necessarily detecting what $f$ or $g$ is}.
\end{quote}

The point is that this method allows the strategy $\mathrm{En}(\mathrm{Id}_{N \Rightarrow N})$ to encode the family $\mathrm{Id}_{N \Rightarrow N}$ without sacrificing its algorithmic nature: The copy-cat of Opponent's strategies $f$ and $g$ is trivially computable, while the potentially infinite plays by $\mathrm{En}(\mathrm{Id}_{N \Rightarrow N}) \bullet \langle f, g \rangle : \mathcal{U}$ faithfully encode whether or not $f = g$. 

In general, positions in $\mathcal{U}$ consist of symbols $\sharp(X)$ that encode type constructions $X$ and ordinary (i.e., not necessarily symbolic) strategies. 
In the following, let us sketch the definition of $\mathcal{U}$.

First, we have to encode \emph{the base cases}, i.e., the games $1$, $0$ and $N$ that interpret One-, Zero- and N-types, respectively, by strategies on the game $\Gamma \Rightarrow \mathcal{U}$.
For this reason, $\mathcal{U}$ subsumes the positions
\begin{mathpar}
\begin{tabular}{c}
$\mathcal{U}$ \\ \hline 
$q$ \\
$\sharp(1)$
\end{tabular}
\and
\begin{tabular}{c}
$\mathcal{U}$ \\ \hline 
$q$ \\
$\sharp(0)$
\end{tabular}
\and
\begin{tabular}{c}
$\mathcal{U}$ \\ \hline 
$q$ \\
$\sharp(N)$
\end{tabular}
\end{mathpar}
so that there are strategies $\mathrm{En}(1), \mathrm{En}(0), \mathrm{En}(N) : \Gamma \Rightarrow \mathcal{U}$ that compute respectively by
\begin{mathpar}
\mathrm{En}(1) : q \mapsto \sharp(1) 
\and
\mathrm{En}(0) : q \mapsto \sharp(0)  
\and
\mathrm{En}(N) : q \mapsto \sharp(N). 
\end{mathpar}

Next, we consider the inductive step to encode Pi- and Sigma-types. 
Assume that a family $A = (A(\gamma))_{\gamma : \Gamma}$ of games $A(\gamma)$ interprets a type $\mathsf{\Gamma \vdash A \ type}$, and a strategy $\mathrm{En}(A) : \Gamma \Rightarrow \mathcal{U}$ interprets the encoding $\mathsf{\Gamma \vdash En(A) : U}$.
For simplicity, let $\mathsf{\Gamma}$ be the empty context; thus, $\Gamma$ is the terminal game $T$ that has only the trivial strategy, and $A$ is identified with a game.
Assume further that a family $B = (B(\alpha))_{\alpha : A}$ of games $B(\alpha)$ interprets a type $\mathsf{x : A \vdash B \ type}$, and a strategy $\mathrm{En}(B) : A \Rightarrow \mathcal{U}$ interprets the encoding $\mathsf{x : A \vdash En(B) : U}$.
Recall that the game semantics \cite{yamada2022game} interprets the Pi-type $\mathsf{\vdash \Pi(A, B) \ type}$ and the Sigma-type $\mathsf{\vdash \Sigma(A, B) \ type}$ by (the singleton families of) the games $\Pi(A, B)$ and $\Sigma(A, B)$, respectively.
Then, there must be strategies $\mathrm{En}(\Pi(A, B)), \mathrm{En}(\Sigma(A, B)) : T \Rightarrow \mathcal{U} \cong \mathcal{U}$ that respectively encode these families. 
For this reason, the game $\mathcal{U}$ also subsumes the positions
\begin{mathpar}
\begin{tabular}{ccc}
&$\mathcal{U}$& \\ \hline
&$q$& \\
&$\sharp(\Pi)$& \\
$a_1$&& \\
$a_2$&& \\
$\vdots$&&
\end{tabular}
\and
\begin{tabular}{ccc}
&$\mathcal{U}$& \\ \hline
&$q$& \\
&$\sharp(\Pi)$& \\
&&$b_1$ \\
&&$b_2$ \\
&&$\vdots$
\end{tabular}
\and
\begin{tabular}{ccc}
&$\mathcal{U}$& \\ \hline
&$q$& \\
&$\sharp(\Sigma)$& \\
$a_1$&& \\
$a_2$&& \\
$\vdots$&&
\end{tabular}
\and
\begin{tabular}{ccc}
&$\mathcal{U}$& \\ \hline
&$q$& \\
&$\sharp(\Sigma)$& \\
&&$b_1$ \\
&&$b_2$ \\
&&$\vdots$
\end{tabular}
\end{mathpar}
where $a_1 a_2 \dots$ are moves played by the strategy $\mathrm{En}(A) : \mathcal{U}$, and $b_1 b_2 \dots$ by the strategy $\mathrm{En}(B) : A \Rightarrow \mathcal{U}$.
In other words, we define the strategy $\mathrm{En}(\Pi(A, B))$ to be the pairing $\langle \mathrm{En}(A) , \mathrm{En}(B) \rangle : \mathcal{U} \mathbin{\&} (A \Rightarrow \mathcal{U})$ prefixed by the moves $q . \sharp (\Pi)$, and similarly for the strategy $\mathrm{En}(\Sigma(A, B))$.
In this way, the game $\mathcal{U}$ enables the encodings of the games $\Pi(A, B)$ and $\Sigma(A, B)$.
Note, however, that the ambient games $\mathcal{U}$ and $A \Rightarrow \mathcal{U}$ for the positions $a_1 a_2 \dots$ and $b_1 b_2 \dots$, respectively, contain the game $\mathcal{U}$ itself.
In particular, the game $A \Rightarrow \mathcal{U}$ is not the game $\mathcal{U}$ itself but the function game from $A$ to $\mathcal{U}$.
Accordingly, this idea necessitates a nontrivial \emph{recursive} definition of the game $\mathcal{U}$.
Our main technical achievement is to realise such a definition, subsuming the general case where the game $\Gamma$ can be different from the trivial one $T$.

Finally, there must be a strategy $\mathrm{En}(\mathrm{Id}_{A}(\langle \alpha, \alpha' \rangle)) : \mathcal{U}$ for each pair $\alpha, \alpha' : A$ of strategies that encodes (the singleton family of) the game $\mathrm{Id}_A(\langle \alpha, \alpha' \rangle)$.
For this reason, we further add the positions
\begin{mathpar}
\begin{tabular}{cccc}
&$\mathcal{U}$&& \\ \hline
&$q$&& \\
&$\sharp(\mathrm{Id})$&& \\
$a_1$&&& \\
$a_2$&&& \\
$\vdots$&&&
\end{tabular}
\and
\begin{tabular}{cccc}
&$\mathcal{U}$&& \\ \hline
&$q$&& \\
&$\sharp(\mathrm{Id})$&& \\
&&$c_1$& \\
&&$c_2$& \\
&&$\vdots$&
\end{tabular}
\and
\begin{tabular}{cccc}
&$\mathcal{U}$&& \\ \hline
&$q$&& \\
&$\sharp(\mathrm{Id})$&& \\
&&&$c'_1$ \\
&&&$c'_2$ \\
&&&$\vdots$
\end{tabular}
\end{mathpar}
where the moves $a_1 a_2 \dots$ are played by the strategy $\mathrm{En}(A) : \mathcal{U}$, the moves $c_1 c_2 \dots$ by the strategy $\alpha : A$ and the moves $c'_1 c'_2 \dots$ by the strategy $\alpha' : A$.
In other words, we define the strategy $\mathrm{Id}_A(\langle \alpha, \alpha' \rangle)$ to be the pairing $\langle \mathrm{En}(A), \langle \alpha, \alpha' \rangle \rangle : \mathcal{U} \mathbin{\&} (A \mathbin{\&} A)$ prefixed by the moves $q . \sharp(\mathrm{Id})$.
It is easy to see how this can be lifted to the general case, where the game $\Gamma$ can be different from the trivial one $T$, and the strategies $\alpha, \alpha'$ are on the game $\Pi(\Gamma, A)$ for Pi-types.
This generalisation is illustrated in the next paragraph. 

Now, let us see how this idea solves the problem sketched in \S\ref{HowToEncodeGamesByStrategies}. 
Instead of the trivial assumption $\Gamma = T$, take $\Gamma = (N \Rightarrow N) \mathbin{\&} (N \Rightarrow N)$, and further let $A$ be the singleton family $\{ N \Rightarrow N \}$, together with the projections $\alpha = \pi_1 : (N \Rightarrow N) \mathbin{\&} (N \Rightarrow N) \rightarrow (N \Rightarrow N)$ and $\alpha' = \pi_2 : (N \Rightarrow N) \mathbin{\&} (N \Rightarrow N) \rightarrow (N \Rightarrow N)$. 
Then, we define the strategy $\mathrm{En}(\mathrm{Id}_{N \Rightarrow N}(\langle \pi_1, \pi_2 \rangle)) : (N \Rightarrow N) \mathbin{\&} (N \Rightarrow N) \rightarrow \mathcal{U}$ to play in either of the following ways illustrated in Figure~\ref{FigExEncodingOfId}, depending on the moves played by Opponent.
\begin{figure}
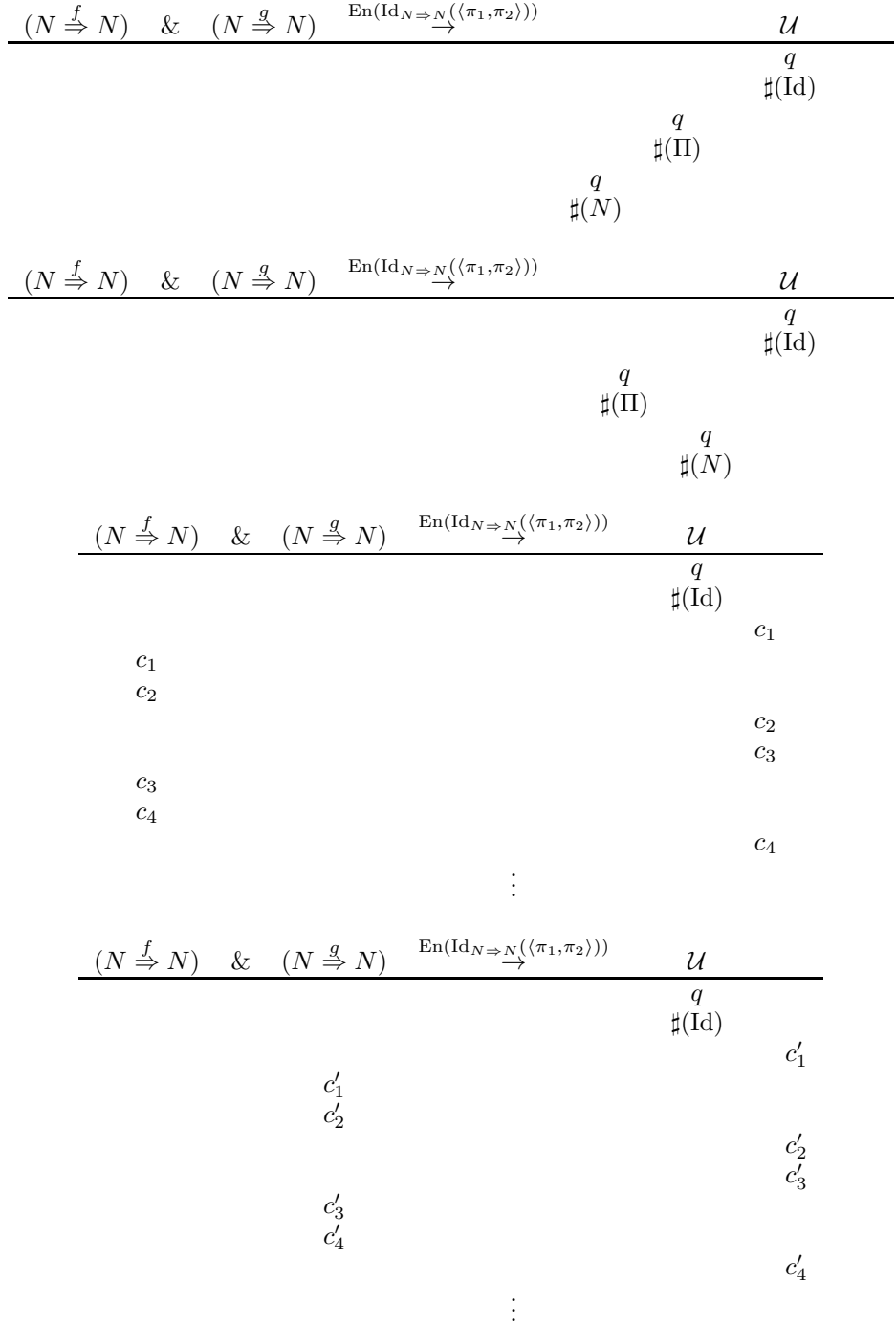

\begin{mathpar}
\begin{tabular}{ccccccccccc}
$(N \stackrel{f}{\Rightarrow} N)$ & $\&$ & $(N \stackrel{g}{\Rightarrow} N)$ & $\stackrel{\mathrm{En}(\mathrm{Id}_{N \Rightarrow N}(\langle \pi_1, \pi_2 \rangle))}{\rightarrow}$ &&& & $\mathcal{U}$ && \\ \hline
&&&&&&& $q$ \\
&&&&&&& $\sharp(\mathrm{Id})$ \\
&&&&& $q$ \\
&&&&& $\sharp(\Pi)$ \\
&&&& $q$ \\
&&&& $\sharp(N)$
\end{tabular}
\and
\begin{tabular}{ccccccccccc}
$(N \stackrel{f}{\Rightarrow} N)$ & $\&$ & $(N \stackrel{g}{\Rightarrow} N)$ & $\stackrel{\mathrm{En}(\mathrm{Id}_{N \Rightarrow N}(\langle \pi_1, \pi_2 \rangle))}{\rightarrow}$ &&& & $\mathcal{U}$ && \\ \hline
&&&&&&& $q$ \\
&&&&&&& $\sharp(\mathrm{Id})$ \\
&&&&& $q$ \\
&&&&& $\sharp(\Pi)$ \\
&&&&&& $q$ \\
&&&&&& $\sharp(N)$
\end{tabular}
\and
\begin{tabular}{ccccccccc}
$(N \stackrel{f}{\Rightarrow} N)$ & $\&$ & $(N \stackrel{g}{\Rightarrow} N)$ & $\stackrel{\mathrm{En}(\mathrm{Id}_{N \Rightarrow N}(\langle \pi_1, \pi_2 \rangle))}{\rightarrow}$ & & $\mathcal{U}$ && \\ \hline
&&&&& $q$ \\
&&&&& $\sharp(\mathrm{Id})$ \\
&&&&&& $c_1$ \\
$c_1$ \\
$c_2$ \\
&&&&&& $c_2$ \\
&&&&&& $c_3$ \\
$c_3$ \\
$c_4$ \\
&&&&&& $c_4$ \\
&&& $\vdots$
\end{tabular}
\and
\begin{tabular}{ccccccccc}
$(N \stackrel{f}{\Rightarrow} N)$ & $\&$ & $(N \stackrel{g}{\Rightarrow} N)$ & $\stackrel{\mathrm{En}(\mathrm{Id}_{N \Rightarrow N}(\langle \pi_1, \pi_2 \rangle))}{\rightarrow}$ & & $\mathcal{U}$ && \\ \hline
&&&&& $q$ \\
&&&&& $\sharp(\mathrm{Id})$ \\
&&&&&&& $c'_1$ \\
&& $c'_1$ \\
&& $c'_2$ \\
&&&&&&& $c'_2$ \\
&&&&&&& $c'_3$ \\
&& $c'_3$ \\
&& $c'_4$ \\
&&&&&&& $c'_4$ \\
&&& $\vdots$
\end{tabular}
\end{mathpar}
\caption{An illustration of the strategy on the encoding of the Id-type between functions}
\label{FigExEncodingOfId}
\end{figure}

In the first two patterns of Figure~\ref{FigExEncodingOfId}, the strategy $\mathrm{En}(\mathrm{Id}_{N \Rightarrow N}(\langle \pi_1, \pi_2 \rangle))$ encodes the underlying family $A = \{ N \Rightarrow N \}$, where recall that function $\Rightarrow$ on games is the trivial class of Pi $\Pi$ on games.
Hence, the family $A$ is encoded simply by the pairing $\langle \mathrm{En}(N), \mathrm{En}(N) \rangle : \mathcal{U} \mathbin{\&} \mathcal{U}$ prefixed by the moves $q . \sharp(\Pi)$.

In the last two patterns of the figure, what the strategy $\mathrm{En}(\mathrm{Id}_{N \Rightarrow N}(\langle \pi_1, \pi_2 \rangle))$ does is essentially to \emph{copy-cat} the input strategies $f$ or $g$ given by Opponent. 
Hence, this strategy is trivially effective, but also its (potentially infinite) plays collectively have the complete information about $f$ and $g$, in particular whether or not $f = g$.
In this way, we overcome the main problem sketched in \S\ref{HowToEncodeGamesByStrategies}.

\subsection{Lifting to the cumulative hierarchy of universes}
The universe $\mathsf{U}$ does not have its own code since otherwise the code $\mathsf{\Gamma \vdash En(U) : U}$ leads to inconsistency known as \emph{Girard's paradox} \cite{girard1972interpretation}.
To address this problem, Martin-L\"{o}f excluded the judgement $\mathsf{\Gamma \vdash En(U) : U}$ and proposed a \emph{cumulative hierarchy} of universes $(\mathsf{U_k})_{k \in \mathbb{N}}$ \cite{martin1975intuitionistic,martin1984intuitionistic}. 
The first universe $\mathsf{U}_0$ does not have its own code $\mathsf{En(U_0)}$, but the second universe $\mathsf{U_1}$ has.
Similarly, the second universe $\mathsf{U}_1$ does not have its own code $\mathsf{En(U_1)}$, but the third universe $\mathsf{U_2}$ has, and so on.
The hierarchy of these universes is \emph{cumulative}: If $i < j$, then the larger universe $\mathsf{U_j}$ has all codes in the smaller one $\mathsf{U_i}$ plus the code $\mathsf{En(U_i)}$.
In this way, the hierarchy collectively encodes \emph{every} type, including the universes themselves, by a code in some universe $\mathrm{U_k}$.
Note that the universe $\mathsf{U}$ is identified with the first universe $\mathsf{U_0}$.

Having established the game $\mathcal{U}$ for the universe $\mathrm{U}$, it is straightforward to interpret the cumulative hierarchy $(\mathsf{U_k})_{k \in \mathbb{N}}$ of universes by a cumulative hierarchy $(\mathcal{U}_k)_{k \in \mathbb{N}}$ of games: For the base case, we define $\mathcal{U}_0 \colonequals \mathcal{U}$; for the inductive step, we define $\mathcal{U}_{k+1}$ by adding the code $\mathrm{En}(\mathcal{U}_i)$ for $i = 0, 1, \dots, k$ to $\mathcal{U}$.


\subsection{Main results}
\label{MainResults}
Based on the idea just sketched, we obtain the following main results of the present work:
\begin{theorem}[computational game semantics of universes]
\label{ThmMainTheorem}
The game semantics of MLTT \cite{yamada2022game} is extendable to the cumulative hierarchy of universes without sacrificing its computability.
\end{theorem}

This theorem in turn extends the independence proof of the previous work \cite{yamada2022game}:
\begin{corollary}[independence of Markov's principle]
\label{CorMainCorollary}
Markov's principle is independent from MLTT equipped with the cumulative hierarchy of universes.
\end{corollary}

This corollary illustrates a strong advantage of game semantics: The combinatorial reasoning of game semantics such as the independence proof remains valid \emph{even when game semantics is extended to new types}.
Hence, when the game semantics of MLTT has been extended to other types, the meta-theoretic results on MLTT shown by the game semantics will be automatically extended to those types as well. 

This advantage makes game semantics a quite powerful tool for the study of MLTT.
In contrast, the syntactic proof given by Coquand and Manna \cite{mannaa2017independence}, for instance, does not have such a modular property because an extension of MLTT may invalidate their syntactic, inductive reasoning.

\subsection{Our contributions and related work}
Our main contribution is the first game semantics of universes (Theorem~\ref{ThmMainTheorem}) in the 30 years history of game semantics.
The main challenge in achieving game semantics of universes is how to encode games by strategies, especially games that interpret Id-types (\S\ref{HowToEncodeGamesByStrategies}).
We solve this problem by the novel idea to encode games by strategies that consist of both symbolic and non-symbolic computations (\S\ref{Solution}), while we allow the decoding function $\mathrm{El}$ to be uncomputable (without sacrificing the effective nature of the game semantics of MLTT \cite{yamada2022game}).
This idea in turn requires a nontrivial recursive definition of games for interpreting universes, and our main technical contribution is to establish such a definition. 

Another contribution is to show the independence of Makov's principle from MLTT equipped with the cumulative hierarchy of universes (Corollary~\ref{CorMainCorollary}). 
This result demonstrates the modular property of the game-semantic reasoning: A meta-theoretic result on MLTT given by the game semantics of MLTT is automatically extended to new types as soon as the game semantics is extended to the types. 

Abramsky et al. \cite{abramsky2015games} establishes the first intensional semantics of a fragment of MLTT.
However, they interpret Sigma-types indirectly by a list construction, not by games, which makes an interpretation of universes hopeless. 
Besides, their method is valid only for a specific class of types \cite[Figure~7]{vakar2018game}, which excludes, e.g., the list type.
Because the list type is constructible by the elimination rule of N-type with respect to universes, this limitation also implies that their approach cannot interpret universes. 

Blot and Laird \cite[Table~3]{blot2018extensional} also interpret a universe, for which they write $\mathsf{\Gamma \vdash_{\mathcal{E}} \mathcal{I} \ type}$, but their interpretation is by domain theory \cite{gierz2003continuous}, not game semantics. 
Besides, they do not interpret Id-types and instead sketch how to interpret Id-types by finite tuples of Boolean-type \cite[\S 9]{blot2018extensional}; however, this method does not work in the presence of N-type since the set $\mathbb{N}$ of all natural numbers is unbounded.

Finally, Coquand and Manna \cite{mannaa2017independence} show the independence of Markov's principle from MLTT equipped with a single universe for the first time in the literature.
Their independence proof is syntactic, which stands in contrast to our game-semantic proof.
As we have mentioned in \S\ref{MainResults}, their syntactic proof is not straightforward to extend to other types, while our game-semantic proof is.

\subsection{The structure of the present article}
The rest of the present article proceeds as follows.
We first prepare for the interpretation of universes by recalling the game semantics of MLTT \cite{yamada2022game} in \S\ref{Review}.
We then proceed to our main contribution in \S\ref{GameSemanticsOfUniverses}: game semantics of the cumulative hierarchy of universes. 
We next present immediate corollaries of this result in \S\ref{Corollaries}, including the independence of Markov's principle from MLTT equipped with the hierarchy of universes. 
We finally draw a conclusion and propose some future work in \S\ref{ConclusionAndFutureWork}.

\begin{notation}
We use the following notations:
\begin{itemize}

\item We use bold small letters $\boldsymbol{s}, \boldsymbol{t}, \boldsymbol{u}, \boldsymbol{v}$, etc. for sequences, in particular $\boldsymbol{\epsilon}$ for the \emph{empty sequence}, and small letters $a, b, m, n, x, y$, etc. for elements of sequences;

\item We define $\overline{n} \colonequals \{ 1, 2, \dots, n \}$ for each $n \in \mathbb{N}_+ \colonequals \mathbb{N} \setminus \{ 0 \}$, and $\overline{0} \colonequals \emptyset$;

\item We write $x_1 x_2 \dots x_{|\boldsymbol{s}|}$ for a sequence $\boldsymbol{s} = (x_1, x_2, \dots, x_{|\boldsymbol{s}|})$, where $|\boldsymbol{s}|$ is the \emph{length} of $\boldsymbol{s}$, define $\boldsymbol{s}(i) \colonequals x_i$ ($i \in \overline{|\boldsymbol{s}|}$) and write $a \in \boldsymbol{s}$ if $a = \boldsymbol{s}(j)$ for some $j \in \overline{|\boldsymbol{s}|}$;


\item A \emph{concatenation} of sequences $\boldsymbol{s}$ and $\boldsymbol{t}$ is represented by their juxtaposition $\boldsymbol{s}\boldsymbol{t}$ (or $\boldsymbol{s} . \boldsymbol{t}$), but we often write $a \boldsymbol{s}$, $\boldsymbol{t} b$, $\boldsymbol{u} c \boldsymbol{v}$ for $(a) \boldsymbol{s}$, $\boldsymbol{t} (b)$, $\boldsymbol{u} (c) \boldsymbol{v}$, and so on;


\item We write $\mathrm{Even}(\boldsymbol{s})$ (resp. $\mathrm{Odd}(\boldsymbol{s})$) if $\boldsymbol{s}$ is of even- (resp. odd-) length, and given a set $S$ of sequences and $\mathrm{P} \in \{ \mathrm{Even}, \mathrm{Odd} \}$, we define $S^\mathrm{P} \colonequals \{ \, \boldsymbol{s} \in S \mid \mathrm{P}(\boldsymbol{s}) \, \}$;

\item We write $\boldsymbol{s} \preceq \boldsymbol{t}$ if $\boldsymbol{s}$ is a \emph{prefix} of a sequence $\boldsymbol{t}$, and given a set $S$ of sequences, $\mathrm{Pref}(S)$ for the set of all prefixes of sequences in $S$, i.e., $\mathrm{Pref}(S) \colonequals \{ \, \boldsymbol{s} \mid \exists \boldsymbol{t} \in S . \, \boldsymbol{s} \preceq \boldsymbol{t} \, \}$.







\end{itemize}
\end{notation}

\section{Review: game semantics of Martin-L\"{o}f type theory}
\label{Review}
In this section, we recall the game semantics of MLTT given in the previous work \cite{yamada2022game}.
To this end, we first recall games and strategies \`{a} la McCusker  \cite{mccusker1998games} (with the slight modifications made by the previous work \cite{yamada2022game}) that interpret simple type theories \cite{abramsky1999game} in \S\ref{GamesAndStrategies} since the previous work is based on this variant of games and strategies.
We then review basic definitions and results of the game semantics of MLTT \cite{yamada2022game} in \S\ref{GameSemanticsOfMLTT}.
Our exposition is minimal; see the tutorial \cite{abramsky1999game} and the previous work \cite{yamada2022game} for more explanations and examples.

\subsection{Games and strategies}
\label{GamesAndStrategies}
\emph{Games} are a class of directed rooted forests. 
For technical convenience, we identify games with the sets of all paths from the roots, called \emph{positions}.
The vertices of games are called \emph{moves}, and positions must be \emph{legal}.
These concepts are centred around the structure of \emph{arenas}.

\begin{definition}[moves \cite{yamada2022game}]
\label{DefMoves}
Let us fix arbitrary pairwise distinct symbols $\mathrm{O}$, $\mathrm{P}$, $\mathrm{Q}$ and $\mathrm{A}$, and call them \emph{labels}.
A \emph{move} is a triple $m^{xy} \colonequals (m, x, y)$ such that $x \in \{ \mathrm{O}, \mathrm{P} \}$ and $y \in \{ \mathrm{Q}, \mathrm{A} \}$. 
We abbreviate moves $m^{xy}$ as $m$ and instead define $\lambda(m) \colonequals xy$, $\lambda^\mathrm{OP}(m) \colonequals x$ and $\lambda^\mathrm{QA}(m) \colonequals y$.
\end{definition}

We call a move $m$ an \emph{O-move} if $\lambda^\mathrm{OP}(m) = \mathrm{O}$, a \emph{P-move} if $\lambda^\mathrm{OP}(m) = \mathrm{P}$, a \emph{question} if $\lambda^\mathrm{QA}(m) = \mathrm{Q}$, and an \emph{answer} if $\lambda^\mathrm{QA}(m) = \mathrm{A}$.

\begin{definition}[arenas \cite{mccusker1998games,yamada2022game,hyland2000full}]
\label{DefArenas}
An \emph{arena} is a pair $G = (M_G, \vdash_G)$ of
\begin{itemize}
\item A set $M_G$ of moves;

\item A subset $\vdash_G$ of the cartesian product $(\{ \star \} \cup M_G) \times M_G$, where $\star$ (also written $\star_G$) is an arbitrarily fixed element such that $\star \not \in M_G$, called the \emph{enabling relation}, that satisfies
\begin{itemize}

\item \textsc{(E1)} If $\star \vdash_G m$, then $\lambda (m) = \mathrm{OQ}$;

\item \textsc{(E2)} If $m \vdash_G n$ and $\lambda^\mathrm{QA} (n) = \mathrm{A}$, then $\lambda^\mathrm{QA} (m) = \mathrm{Q}$;

\item \textsc{(E3)} If $m \vdash_G n$ and $m \neq \star$, then $\lambda^\mathrm{OP} (m) \neq \lambda^\mathrm{OP} (n)$.

\end{itemize}

\end{itemize}
\end{definition}

We call moves $m \in M_G$ \emph{initial} if $\star \vdash_G m$, and define the set $M_G^{\mathrm{Init}} \colonequals \{ \, m \in M_G \mid \star \vdash_G m \, \}$ of all initial moves of $G$.
An arena $G$ is \emph{well-founded} if the relation $\vdash_G$ is well-founded, i.e., there is no sequence $(m_i)_{i \in \mathbb{N}}$ of moves $m_i \in M_G$ such that $\star \vdash_G m_0$ and $m_i \vdash_G m_{i+1}$ for all $i \in \mathbb{N}$.

Strictly speaking, positions of games are sequences of moves equipped with \emph{pointers}:
\begin{definition}[j-sequences \cite{hyland2000full,coquand1995semantics,mccusker1998games}]
\label{DefJSequences}
An \emph{occurrence} in a finite sequence $\boldsymbol{s}$ is a pair $(\boldsymbol{s}(i), i)$ such that $i \in \overline{|\boldsymbol{s}|}$. 
A \emph{justified (j-) sequence} is a pair $\boldsymbol{s} = (\boldsymbol{s}, \mathcal{J}_{\boldsymbol{s}})$ of a finite sequence $\boldsymbol{s}$ of moves and a map $\mathcal{J}_{\boldsymbol{s}} : \overline{|\boldsymbol{s}|} \rightarrow \{ 0 \} \cup \overline{|\boldsymbol{s}|-1}$ such that $0 \leqslant \mathcal{J}_{\boldsymbol{s}}(i) < i$ for all $i \in \overline{|\boldsymbol{s}|}$, called the \emph{pointer} of the j-sequence. 
An occurrence $(\boldsymbol{s}(i), i)$ is \emph{initial} in $\boldsymbol{s}$ if $\mathcal{J}_{\boldsymbol{s}}(i) = 0$.
\end{definition}

We say that the occurrence $(\boldsymbol{s}({\mathcal{J}_{\boldsymbol{s}}(i)}), \mathcal{J}_{\boldsymbol{s}}(i))$ is the \emph{justifier} of a non-initial one $(\boldsymbol{s}(i), i)$ in $\boldsymbol{s}$, and $(\boldsymbol{s}(i), i)$ is \emph{justified} by $(\boldsymbol{s}({\mathcal{J}_{\boldsymbol{s}}(i)}), \mathcal{J}_{\boldsymbol{s}}(i))$ in $\boldsymbol{s}$. 
A j-sequence $\boldsymbol{s}$ is \emph{in an arena} $G$ if its elements are moves of $G$, and its pointer respects the relation $\vdash_G$ in $G$, i.e., $\forall i \in \overline{|\boldsymbol{s}|} . \, \big(\mathcal{J}_{\boldsymbol{s}}(i) = 0 \Rightarrow \star \vdash_G \boldsymbol{s}(i)\big) \wedge \big(\mathcal{J}_{\boldsymbol{s}}(i) \neq 0 \Rightarrow \boldsymbol{s}({\mathcal{J}_{\boldsymbol{s}}(i)}) \vdash_G \boldsymbol{s}(i)\big)$.
We write $\mathscr{J}_G$ for the set of all j-sequences in $G$.
A \emph{justified (j-) subsequence} of a j-sequence $\boldsymbol{s}$ is a j-sequence $\boldsymbol{t}$, written $\boldsymbol{t} \sqsubseteq \boldsymbol{s}$, such that $\boldsymbol{t}$ is a subsequence of $\boldsymbol{s}$, and $\mathcal{J}_{\boldsymbol{t}}(i) = j$ if and only if $\mathcal{J}_{\boldsymbol{s}}^n(i) = j$ for some $n \in \mathbb{N}_+$ with the occurrences $(\boldsymbol{s}(\mathcal{J}_{\boldsymbol{s}}^k(i)), \mathcal{J}_{\boldsymbol{s}}^k(i))$ for $k = 1, 2, \dots, n-1$ deleted in $\boldsymbol{t}$.

\begin{convention}
We are henceforth casual about the distinction between moves and occurrences, and by abuse of notation, we frequently keep the pointer $\mathcal{J}_{\boldsymbol{s}}$ of each j-sequence $\boldsymbol{s} = (\boldsymbol{s}, \mathcal{J}_{\boldsymbol{s}})$ implicit since it is mostly obvious, and abbreviate occurrences $(\boldsymbol{s}(i), i)$ in $\boldsymbol{s}$ as $\boldsymbol{s}(i)$.
We write $\mathcal{J}_{\boldsymbol{s}}(\boldsymbol{s}(i)) = \boldsymbol{s}(j)$ if $\mathcal{J}_{\boldsymbol{s}}(i) = j > 0$.
\end{convention}

\begin{definition}[views \cite{coquand1995semantics,hyland2000full,mccusker1998games}] 
\label{DefViews}
The \emph{P-view} $\lceil \boldsymbol{s} \rceil$ and the \emph{O-view} $\lfloor \boldsymbol{s} \rfloor$ of a j-sequence $\boldsymbol{s}$ are the j-subsequences of $\boldsymbol{s}$ defined by the induction 
\begin{itemize}

\item $\lceil \boldsymbol{\epsilon} \rceil \colonequals \boldsymbol{\epsilon}$; 

\item $\lceil \boldsymbol{s} m \rceil \colonequals \lceil \boldsymbol{s} \rceil . m$ if $m$ is a P-move; 

\item $\lceil \boldsymbol{s} m \rceil \colonequals m$ if $m$ is initial;

\item $\lceil \boldsymbol{s} m \boldsymbol{t} n \rceil \colonequals \lceil \boldsymbol{s} \rceil . m n$ if $n$ is an O-move such that $m$ justifies $n$; 

\item $\lfloor \boldsymbol{\epsilon} \rfloor \colonequals \boldsymbol{\epsilon}$;

\item $\lfloor \boldsymbol{s} m \rfloor \colonequals \lfloor \boldsymbol{s} \rfloor . m$ if $m$ is an O-move; 

\item $\lfloor \boldsymbol{s} m \boldsymbol{t} n \rfloor \colonequals \lfloor \boldsymbol{s} \rfloor . m n$ if $n$ is a P-move such that $m$ justifies $n$. 

\end{itemize}
\end{definition}

\begin{definition}[legal positions \cite{abramsky1999game,mccusker1998games}]
\label{DefLegalPositions}
A \emph{legal position} is a j-sequence $\boldsymbol{s}$ such~that
\begin{itemize}

\item \textsc{(Alternation)} If $\boldsymbol{s} = \boldsymbol{s_1} m n \boldsymbol{s_2}$, then $\lambda^\mathrm{OP} (m) \neq \lambda^\mathrm{OP} (n)$;


\item \textsc{(Visibility)} If $\boldsymbol{s} = \boldsymbol{t} m \boldsymbol{u}$ with $m$ non-initial, then $\mathcal{J}_{\boldsymbol{s}}(m)$ occurs in the P-view $\lceil \boldsymbol{t} \rceil$ if $m$ is a P-move, and in the O-view $\lfloor \boldsymbol{t} \rfloor$ otherwise.

\end{itemize}
\end{definition}

A legal position is \emph{in an arena} $G$ if it is a j-sequence in $G$ (Definition~\ref{DefJSequences}).
We write $\mathscr{L}_G$ for the set of all legal positions in $G$.

\begin{definition}[games \cite{mccusker1998games,abramsky1999game,yamada2022game}]
\label{DefGames}
A \emph{game} is a set $G$ of legal positions such that
\begin{enumerate}

\item $G$ is nonempty and \emph{prefix-closed} (i.e., $\boldsymbol{s}m \in G \Rightarrow \boldsymbol{s} \in G$);

\item $\mathrm{Arn}(G) \colonequals (M_G, \vdash_G)$ is an arena, where $M_G \colonequals \{ \, \boldsymbol{s}(i) \mid \boldsymbol{s} \in G, i \in \overline{|\boldsymbol{s}|} \, \}$ and $\vdash_G \, \colonequals \{ \, (\star, \boldsymbol{s}(j)) \mid \boldsymbol{s} \in G, \mathcal{J}_{\boldsymbol{s}}(j) = 0 \, \}\cup \{ \, (\boldsymbol{s}(i), \boldsymbol{s}(j)) \mid \boldsymbol{s} \in S, \mathcal{J}_{\boldsymbol{s}}(j)= i > 0 \, \}$.

\end{enumerate}
\end{definition}

A game $G$ is \emph{well-founded} if so is the arena $\mathrm{Arn}(G)$, and \emph{well-opened} if each of its elements has at most one initial occurrence (i.e., the conjunction of $\boldsymbol{s}m \in G$ and $m \in M_G^{\mathrm{Init}}$ implies $\boldsymbol{s} = \boldsymbol{\epsilon}$).
We call elements of $G$ \emph{(valid) positions} in $G$.
A \emph{subgame} of $G$ is a game $H \subseteq G$, and $\mathrm{sub}(G) \colonequals \{ \, H \mid \text{$H$ is a subgame of $G$} \, \}$.

\begin{example}
\label{ExamplesOfGames}
The simplest game is the \emph{terminal game} $T \colonequals \{ \boldsymbol{\epsilon} \}$ which only has the trivial position $\boldsymbol{\epsilon}$.
The \emph{flat game} on a given set $S$ is the game $\mathrm{flat}(S) \colonequals \mathrm{Pref}(\{ \, q^{\mathrm{OQ}} . m^{\mathrm{PA}} \mid m \in S \, \})$, where $q$ is an arbitrarily fixed element such that $q \not\in S$, and $q^{\mathrm{OQ}}$ justifies $m^{\mathrm{PA}}$. 
Consider, for instance, the \emph{empty game} $0 \colonequals \mathrm{flat}(\emptyset)$ and the \emph{natural number game} $N \colonequals \mathrm{flat}(\mathbb{N})$.
As the notation indicates, the empty game $0$ interprets Zero-type, and the natural number game $N$ interprets N-type (\S\ref{GameSemanticsOfMLTT}).
\end{example}

On the other hand, \emph{strategies} on a game $G$ are algorithms for Player about how to play on $G$:
\begin{definition}[strategies \cite{mccusker1998games}]
\label{DefStrategies}
A \emph{strategy} on a game $G$ is a subset $\sigma \subseteq G^{\mathrm{Even}}$, written $\sigma : G$, that is nonempty, \emph{even-prefix-closed} (i.e., $\boldsymbol{s}mn \in \sigma \Rightarrow \boldsymbol{s} \in \sigma$) and \emph{deterministic} (i.e., $\boldsymbol{s}mn, \boldsymbol{s}mn' \in \sigma \Rightarrow \boldsymbol{s}mn = \boldsymbol{s}mn'$).
We write $\mathrm{st}(G)$ for the set $\{ \, \sigma \mid \sigma : G \, \}$ of all strategies on $G$.
\end{definition}

The idea is that a strategy $\sigma : G$ describes for Player how to play on the game $G$ by the computation $\boldsymbol{s}m \in G^{\mathrm{Odd}} \mapsto \boldsymbol{s}mn \in \sigma$ (n.b., $m$ is an O-move, and $n$ is a P-move), if any, which is \emph{deterministic} by the determinacy of $\sigma$, and in general \emph{partial} since there can be no output $\boldsymbol{s}mn \in \sigma$.

\begin{example}
\label{ExamplesOfStrategies}
The terminal game $T$ only has the trivial strategy $\top \colonequals \{ \boldsymbol{\epsilon} \}$, and the flat game $\mathrm{flat}(S)$ on a given set $S$ has strategies $\bot \colonequals \{ \boldsymbol{\epsilon} \}$ and $\underline{m} \colonequals \{ \boldsymbol{\epsilon}, q m \}$ for each $m \in S$. 
\end{example}

Strategies are \emph{unrestricted} computations, e.g., they can be \emph{partial}, some of which do not correspond to proofs in logic or formal systems.
This motivates \emph{winning} and \emph{well-bracketing} on strategies: Winning strategies correspond to proofs in classical logic, and winning, well-bracketed ones to proofs in intuitionistic logic.
Because the underlying logic of MLTT is intuitionistic, we achieve a tight correspondence between MLTT and game semantics by focusing on winning, well-bracketed strategies.
\begin{definition}[constraints on strategies \cite{coquand1995semantics,laird1997full,mccusker1998games,abramsky1999game}]
\label{DefConstraintsOnStrategies}
A strategy $\sigma : G$ is
\begin{itemize}

\item \emph{Total} if it always responds: $\forall \boldsymbol{s} \in \sigma, \boldsymbol{s} m \in G . \, \exists \boldsymbol{s} m n \in \sigma$;

\item \emph{Innocent} if it only depends on P-views: $\forall \boldsymbol{s} m n \in \sigma, \boldsymbol{t} l \in G . \, \lceil \boldsymbol{s} m \rceil = \lceil \boldsymbol{t} l \rceil \Rightarrow \exists \boldsymbol{t}lr \in \sigma . \, \lceil \boldsymbol{s} m n \rceil = \lceil \boldsymbol{t}lr \rceil$;

\item \emph{Noetherian} if there is no strictly increasing (with respect to the prefix relation $\preceq$) infinite sequence of elements in the set $\lceil \sigma \rceil \colonequals \{ \, \lceil \boldsymbol{s} \rceil \mid \boldsymbol{s} \in \sigma \, \}$ of all P-views in $\sigma$;

\item \emph{Winning} if it is total, innocent and noetherian;

\item \emph{Well-bracketed} if its `question-answering' in P-views is in the `last-question-first-answered' fashion: If $\boldsymbol{s} q \boldsymbol{t} a \in \sigma$, where $\lambda^{\mathrm{QA}}(q) = \mathrm{Q}$, $\lambda^{\mathrm{QA}}(a) = \mathrm{A}$ and $\mathcal{J}_{\boldsymbol{s}q\boldsymbol{t}a}(a) = q$, then each question occurring in $\boldsymbol{t'}$, where the P-view $\lceil \boldsymbol{s} q \boldsymbol{t} \rceil$ has $\lceil \boldsymbol{s} q \boldsymbol{t} \rceil = \lceil \boldsymbol{s} q \rceil . \boldsymbol{t'}$ by visibility, justifies an answer occurring in $\boldsymbol{t'}$.

\end{itemize}
\end{definition}

\begin{example}
The strategies $\top : T$ and $\underline{n} : N$ for all $n \in \mathbb{N}$ are winning and well-bracketed, while the strategies $\bot : 0$ and $\bot : N$ are not even total, let alone winning. 
\end{example}

Let us next recall standard constructions on games and strategies.
\begin{convention}
For brevity and readability, we omit `tags' for disjoint union $\uplus$. 
For instance, we write $x \in A \uplus B$ if $x \in A$ or $x \in B$; also, given relations $R_A \subseteq A \times A$ and $R_B \subseteq B \times B$, we write $R_A \uplus R_B$ for the relation on $A \uplus B$ such that $(x, y) \in R_A \uplus R_B \ratio \Leftrightarrow (x, y) \in R_A \vee (x, y) \in R_B$. 
\end{convention}

\begin{definition}[constructions on arenas \cite{mccusker1998games}]
\label{DefConstructionsOnArenas}
Given arenas $A$ and $B$, we define
\begin{itemize}

\item $A \uplus B \colonequals (M_A \uplus M_B, \vdash_A \uplus \vdash_B)$; 

\item $A \multimap B \colonequals (\{ \, a^{(x^\bot)y} \mid a^{xy} \in M_A \, \} \uplus M_B, \vdash_{A \multimap B})$, $\mathrm{O}^\bot \colonequals \mathrm{P}$, $\mathrm{P}^\bot \colonequals \mathrm{O}$, $\star \vdash_{A \multimap B} m :\Leftrightarrow \star \vdash_B m$ and $m \vdash_{A \multimap B} n :\Leftrightarrow m \vdash_A n \vee m \vdash_B n \vee (\star \vdash_B m \wedge \star \vdash_A n)$.


\end{itemize}
\end{definition}

\begin{definition}[constructions on games \cite{mccusker1998games}]
\label{DefConstructionsOnGames}
Given games $G$ and $H$, we define
\begin{itemize}

\item The \emph{tensor} $G \otimes H \colonequals \{ \, \boldsymbol{s} \in \mathscr{L}_{\mathrm{Arn}(G) \uplus \mathrm{Arn}(H)} \mid \forall X \in \{ G, H \} . \, \boldsymbol{s} \upharpoonright X \in X \, \}$ of $G$ and $H$, where $\boldsymbol{s} \upharpoonright X \sqsubseteq \boldsymbol{s}$ consists of occurrences of moves in $X$;

\item The \emph{exponential} $\oc G \colonequals \{ \, \boldsymbol{s} \in \mathscr{L}_{\mathrm{Arn}(G)} \mid \forall i \in |\boldsymbol{s}| . \, \mathcal{J}_{\boldsymbol{s}}(i) = 0 \Rightarrow \boldsymbol{s} \upharpoonright \{ (\boldsymbol{s}(i), i) \} \in G \, \}$ of $G$, where $\boldsymbol{s} \upharpoonright \{ (\boldsymbol{s}(i), i) \} \sqsubseteq \boldsymbol{s}$ consists of occurrence in $\boldsymbol{s}$ hereditarily justified by the initial one $(\boldsymbol{s}(i), i)$ in $\boldsymbol{s}$;

\item The \emph{product} $G \mathbin{\&} H \colonequals \{ \, \boldsymbol{s} \in \mathscr{L}_{\mathrm{Arn}(G) \mathbin{\uplus} \mathrm{Arn}(H)} \mid (\boldsymbol{s} \upharpoonright G \in G \wedge \boldsymbol{s} \upharpoonright H = \boldsymbol{\epsilon}) \vee (\boldsymbol{s} \upharpoonright G = \boldsymbol{\epsilon} \wedge \boldsymbol{s} \upharpoonright H \in H) \, \}$ of $G$ and $H$;

\item The \emph{linear implication} $G \multimap H \colonequals \{ \, \boldsymbol{s} \in \mathscr{L}_{\mathrm{Arn}(G) \multimap \mathrm{Arn}(H)} \mid \boldsymbol{s} \upharpoonright G^\bot \in G, \boldsymbol{s} \upharpoonright H \in H \, \}$ from $G$ to $H$, also written $H^G$, where $\boldsymbol{s} \upharpoonright G^\bot$ is obtained from $\boldsymbol{s} \upharpoonright G$ by modifying all the moves $m^{(x^\bot)y}$ occurring in $\boldsymbol{s} \upharpoonright G$ into $m^{xy}$;

\item The \emph{implication} $G \Rightarrow H \colonequals \oc G \multimap H$ from $G$ to $H$.

\end{itemize}

Notationally, exponential $\oc$ precedes other constructions on games, while tensor $\otimes$ and product $\mathbin{\&}$ do linear implication $\multimap$ and implication $\Rightarrow$. 
\end{definition}

\begin{definition}[constructions on strategies \cite{mccusker1998games}]
\label{DefConstructionsOnStrategies}
Given strategies $\phi : A \multimap B$, $\sigma : C \multimap D$, $\tau : A \multimap C$, $\psi : B \multimap C$ and $\theta : \oc A \multimap B$, we define
\begin{itemize}

\item The \emph{copy-cat} $\mathrm{cp}_A \colonequals \{ \, \boldsymbol{s} \in (A_{[0]} \multimap A_{[1]})^{\mathrm{Even}} \mid \forall \boldsymbol{t} \preceq \boldsymbol{s} . \, \mathrm{Even}(\boldsymbol{t}) \Rightarrow \boldsymbol{t} \upharpoonright A_{[0]}^\bot = \boldsymbol{t} \upharpoonright A_{[1]} \, \}$ on $A$; 

\item The \emph{dereliction} $\mathrm{der}_A \colonequals \{ \, \boldsymbol{s} \in (\oc A \multimap A)^{\mathrm{Even}} \mid \forall \boldsymbol{t} \preceq \boldsymbol{s} . \, \mathrm{Even}(\boldsymbol{t}) \Rightarrow \boldsymbol{t} \upharpoonright \oc A^\bot = \boldsymbol{t} \upharpoonright A \, \}$ on $A$;

\item The \emph{tensor} $\phi \otimes \sigma \colonequals \{ \, \boldsymbol{s} \in A \otimes C \multimap B \otimes D \mid \boldsymbol{s} \upharpoonright A, B \in \phi, \boldsymbol{s} \upharpoonright C, D \in \sigma \, \}$ of $\phi$ and $\sigma$, where $\boldsymbol{s} \upharpoonright A, B \sqsubseteq \boldsymbol{s}$ (resp. $\boldsymbol{s} \upharpoonright C, D \sqsubseteq \boldsymbol{s}$) consists of occurrences of moves in $A$ or $B$ (resp. $C$ or $D$);

\item The \emph{pairing} $\langle \phi, \tau \rangle \colonequals \{ \, \boldsymbol{s} \in A \multimap B \mathbin{\&} C \mid (\boldsymbol{s} \upharpoonright A, B \in \phi \wedge \boldsymbol{s} \upharpoonright C = \boldsymbol{\epsilon}) \vee (\boldsymbol{s} \upharpoonright A, C \in \tau \wedge \boldsymbol{s} \upharpoonright B = \boldsymbol{\epsilon}) \, \}$ of $\phi$ and $\tau$;

\item The \emph{composition} $\phi ; \psi \colonequals \{ \, \boldsymbol{s} \upharpoonright A, C \mid \boldsymbol{s} \in \phi \parallel \psi \, \}$ of $\phi$ and $\psi$ (n.b., $\phi; \psi$ is also written $\psi \circ \phi$), where $\phi \parallel \psi \colonequals \{ \, \boldsymbol{s} \in \mathscr{J} \mid \boldsymbol{s} \upharpoonright A, B_{[0]} \in \phi, \boldsymbol{s} \upharpoonright B_{[1]}, C \in \psi, \boldsymbol{s} \upharpoonright B_{[0]}^\bot, B_{[1]}^\bot \in \mathrm{cp}_B \, \}$, $\mathscr{J} \colonequals \mathscr{J}_{\mathrm{Arn}(((A \multimap B_{[0]}) \multimap B_{[1]}) \multimap C)}$, $\boldsymbol{s} \upharpoonright B_{[0]}^\bot, B_{[1]}^\bot$ is obtained from $\boldsymbol{s} \upharpoonright B_{[0]}, B_{[1]}$ by applying the operation $(\_)^\bot : m^{xy} \mapsto m^{x^\bot y}$ (Definition~\ref{DefConstructionsOnArenas}) on all moves $m^{xy}$;

\item The \emph{promotion} $\theta^\dagger \colonequals \{ \, \boldsymbol{s} \in (\oc A \multimap \oc B)^{\mathrm{Even}} \mid \forall i \in |\boldsymbol{s}| . \, \mathcal{J}_{\boldsymbol{s}}(i) = 0 \Rightarrow \boldsymbol{s} \upharpoonright \{ (\boldsymbol{s}(i), i) \} \in \theta \, \}$ of $\theta$. 
\end{itemize}
\end{definition}

\begin{example}
The promotion $\mathrm{succ}^\dagger : \oc N \multimap \oc N$ of the strategy 
\begin{equation*}
\mathrm{succ} \colonequals \{ \, q_{[1]} . q_{[0]} . n_{[0]} . n+1_{[1]} \mid n \in \mathbb{N} \, \} : N_{[0]} \Rightarrow N_{[1]}
\end{equation*}
computes as sketched in the introduction (\S\ref{ExamplesOfGamesAndStrategies}).
\end{example}

Let us summarise the present section by:
\begin{definition}[categories of games \cite{mccusker1998games,yamada2022game}]
\label{DefCategoriesOfGamesAndStrategies}
The category $\mathbb{G}_\oc$ consists of 
\begin{itemize}

\item Well-opened games as objects;

\item Strategies on the implication $A \Rightarrow B$ as morphisms $A \rightarrow B$;

\item The composition $\psi \bullet \phi \colonequals \psi \circ \phi^\dagger : A \Rightarrow C$ of strategies as the composition of morphisms $\phi : A \rightarrow B$ and $\psi : B \rightarrow C$;

\item The dereliction $\mathrm{der}_A$ as the identity on each object $A$.

\end{itemize}

The subcategory $\mathbb{LG}_\oc$ (resp. $\mathbb{WG}$) of $\mathbb{G}_\oc$ consists of well-founded, well-opened games as objects, and winning (resp. winning, well-bracketed) strategies as morphisms.
\end{definition}

We have to focus on \emph{well-opened} games in these categories since otherwise the identities would not be well-defined \cite[pp.~42--43]{mccusker1998games}.
We use the subscript $(\_)_\oc$ in order to distinguish these categories from the \emph{linear} ones, in which morphisms $A \rightarrow B$ are strategies on the linear implication $A \multimap B$.

\begin{notation}
We are not bothered about the distinction between strategies on games $G$ and $T \Rightarrow G$.
\end{notation}

\subsection{Game semantics of Martin-L\"{o}f type theory}
\label{GameSemanticsOfMLTT}
The previous work \cite{yamada2022game} establishes game semantics of MLTT based on games and strategies recalled in the previous section.
The central idea of the precious work is to generalise games into \emph{predicate (p-) games}, which corresponds to the generalisation of simple types to dependent types:
\begin{definition}[p-games \cite{yamada2022game}]
\label{DefPredicateGames}
A \emph{predicate (p-) game} is a pair $\Gamma = (|\Gamma|, \| \Gamma \|)$ of a game $|\Gamma|$ and a family $\| \Gamma \| = (\Gamma(\gamma))_{\gamma : |\Gamma|}$ of subgames $\Gamma(\gamma) \subseteq |\Gamma|$.
It is \emph{well-founded} (resp. \emph{well-opened}) if so is $|\Gamma|$.
\end{definition}

\begin{example}
\label{ExPGames}
Given a game $G$, we have the p-game $\mathscr{P}(G) \colonequals (G, \kappa_G)$, where $\kappa_G$ is the constant family at $G$.
Clearly, $G$ and $\mathscr{P}(G)$ are essentially the same. 
We abbreviate $\mathscr{P}(T)$, $\mathscr{P}(0)$ and $\mathscr{P}(N)$ as $T$, $0$ and $N$, and call them the \emph{terminal p-game}, the \emph{empty p-game} and the \emph{natural number p-game}, respectively.
\end{example}

Before recalling strategies on p-games, we need a few preliminary concepts:
\begin{definition}[liveness ordering \cite{chroboczek2000game}]
\label{DefLivenessOrdering}
The \emph{liveness ordering} is a partial order $\preccurlyeq$ between games \cite[Definition~8 and Theorem~9]{chroboczek2000game}, which defines $G \preccurlyeq H$ to mean that O (resp. P) is less (resp. more) restricted in $G$ than in $H$, i.e., they satisfy
\begin{enumerate}

\item If $\boldsymbol{s} \in (G \cap H)^{\mathrm{Even}}$ and $\boldsymbol{s}m \in H^{\mathrm{Odd}}$, then $\boldsymbol{s}m \in G^{\mathrm{Odd}}$;

\item If $\boldsymbol{t}l \in (G \cap H)^{\mathrm{Odd}}$ and $\boldsymbol{t}lr \in G^{\mathrm{Even}}$, then $\boldsymbol{t}lr \in H^{\mathrm{Even}}$.

\end{enumerate}
\end{definition}

\begin{definition}[closures of strategies \cite{yamada2022game}]
The \emph{closure} of a strategy $\sigma : G$ with respect to another game $H$ is the subgame $\overline{\sigma}_H \colonequals \{ \boldsymbol{\epsilon} \} \cup \{ \, \boldsymbol{s}m \in H^{\mathrm{Odd}} \mid \boldsymbol{s} \in \overline{\sigma}_H \, \} \cup \{ \, \boldsymbol{t}lr \in \sigma \mid \boldsymbol{t}l \in \overline{\sigma}_H \, \} \subseteq \sigma \cup H$.
\end{definition}

We see by induction that $\overline{\sigma}_G = \sigma \cup \{ \, \boldsymbol{s}m \in G \mid \boldsymbol{s} \in \sigma \, \}$ holds for all strategies $\sigma : G$.
Moreover:
\begin{proposition}[liveness characterisation \cite{yamada2022game}]
\label{PropLivenessCharacterisation}
Assume $\sigma : G$ and $H \in \mathrm{sub}(G)$.
\begin{enumerate}

\item $\overline{\sigma}_H^{\mathrm{Even}} : H$ if and only if $\overline{\sigma}_{G} \preccurlyeq H$;

\item If $\overline{\sigma}_{G} \preccurlyeq H$, then $\overline{\sigma}_{H}^{\mathrm{Even}} = \sigma \cap H$.

\end{enumerate}
\end{proposition}

This proposition enables us to define strategies on p-games as follows: 
\begin{definition}[strategies on p-games \cite{yamada2022game}]
\label{DefStrategiesOnPredicateGames}
A \emph{strategy} on a p-game $\Gamma$, written $\gamma : \Gamma$, is a strategy $\gamma : |\Gamma|$ such that $\overline{\gamma}_{|\Gamma|} \preccurlyeq \Gamma(\gamma)$.
It is \emph{total} (resp. \emph{innocent}, \emph{noetherian}, \emph{well-bracketed}) if so is $\gamma \cap \Gamma(\gamma) : \Gamma(\gamma)$.
\end{definition}

We write $\mathrm{st}(\Gamma)$ for the set $\{ \, \gamma \mid \gamma : \Gamma \, \}$ of all strategies on a p-game $\Gamma$ and define $\overline{\gamma}_\Gamma \colonequals \overline{\gamma}_{\Gamma(\gamma)}$ for all $\gamma : \Gamma$.
A \emph{position} in $\Gamma$ is a prefix of a sequence $q_\Gamma \gamma \boldsymbol{s}$ such that $\gamma : \Gamma$ and $\boldsymbol{s} \in \overline{\gamma}_\Gamma$, where $q_\Gamma$ is an arbitrarily fixed element such that $q_\Gamma \not\in M_{|\Gamma|}$, $q_\Gamma \gamma$ is called an \emph{initial protocol}, and $\boldsymbol{s}$ is called an \emph{actual position}.

A play in $\Gamma$ proceeds as follows.
First, \emph{Judge} asks Player a question $q_\Gamma$ (`What is your strategy?') and she answers it by a strategy $\gamma : \Gamma$ (`It is $\gamma$!').
After this initial protocol, an ordinary play on the game $\Gamma(\gamma)$ between Player and Opponent follows, in which Player must use the declared one $\gamma$ restricted to $\Gamma(\gamma)$, i.e., $\gamma \cap \Gamma(\gamma) = \overline{\gamma}_{\Gamma(\gamma)}^{\mathrm{Even}} : \Gamma(\gamma)$.
Thus, $\gamma : \Gamma$ is \emph{winning} (resp. \emph{well-bracketed}) if so is $\overline{\gamma}_{\Gamma(\gamma)}^{\mathrm{Even}} : \Gamma(\gamma)$.

Judge and the initial protocol are mere devices for requiring Player to \emph{fix} the strategy $\gamma : |\Gamma|$ and the game $\Gamma(\gamma)$ that \emph{pass the test} $\overline{\gamma}_{\Gamma(\gamma)}^{\mathrm{Even}} : \Gamma(\gamma)$.
See the beginning of \cite[\S 3]{yamada2022game} for an illustration of how and why these strategy filtering and fixing are necessary for an interpretation of MLTT.

We next recall basic constructions on p-games:
\begin{notation}
Let $G$ be a game, $\boldsymbol{s} \in \oc G$ and $i \in \mathbb{N}$. 
We write $\boldsymbol{s} \upharpoonright i$ for the j-subsequence of $\boldsymbol{s}$ that consists of occurrences hereditarily justified by the $(i+1)$st initial occurrence in $\boldsymbol{s}$.
For instance, if $\boldsymbol{s} = q2q1q0 \in \oc N$, then $\boldsymbol{s} \upharpoonright 0 = q2$, $\boldsymbol{s} \upharpoonright 1 = q1$ and $\boldsymbol{s} \upharpoonright 2 = q0$.

Given a strategy $\sigma$ on the tensor $G_0 \otimes G_1$ of games $G_i$ ($i = 0, 1$), we define 
\begin{equation}
\sigma \upharpoonright G_i \colonequals \begin{cases} \sigma_i &\text{if $\sigma = \sigma_0 \otimes \sigma_1$ for (necessarily unique) $\sigma_0 : G_0$ and $\sigma_1 : G_1$;} \\ \uparrow &\text{otherwise, where $\uparrow$ means being \emph{undefined}.} \end{cases}
\end{equation}

Similarly, given a strategy $\tau$ on the exponential $\oc G$ of a game $G$ and $j \in \mathbb{N}$, we define
\begin{equation}
\tau \upharpoonright j \colonequals \begin{cases} \{ \, \boldsymbol{s} \upharpoonright j \mid \boldsymbol{s} \in \tau \, \} &\text{if $\{ \, \boldsymbol{s} \upharpoonright k \mid \boldsymbol{s} \in \tau \, \} : G$ for all $k \in \mathbb{N}$;} \\ \uparrow &\text{otherwise.} \end{cases}
\end{equation}

Given a p-game $\Gamma$, we define the value $\Gamma(\uparrow)$ to be \emph{undefined}, and the constructions $\otimes$, $\multimap$, $\&$ and $\oc$ on undefined games to be \emph{undefined}.
Finally, we extend the relation $\overline{\gamma}_{|\Gamma|} \preccurlyeq \Gamma(\gamma)$ by defining that \emph{it does not hold if the game $\Gamma(\gamma)$ is undefined}.
\end{notation}

\begin{definition}[product and tensor on p-games \cite{yamada2022game}]
The \emph{product} of p-games $\Gamma$ and $\Delta$ is the p-game $\Gamma \mathbin{\&} \Delta$ defined by $|\Gamma \mathbin{\&} \Delta| \colonequals |\Gamma| \mathbin{\&} |\Delta|$ and $(\Gamma \mathbin{\&} \Delta)(\langle \gamma, \delta \rangle) \colonequals \Gamma(\gamma) \mathbin{\&} \Delta(\delta)$ for all $\langle \gamma, \delta \rangle : |\Gamma \mathbin{\&} \Delta|$, and their \emph{tensor} is the p-game $\Gamma \otimes \Delta$ defined by $|\Gamma \otimes \Delta| \colonequals |\Gamma| \otimes |\Delta|$ and $(\Gamma \otimes \Delta)(\sigma) \colonequals \Gamma(\sigma \upharpoonright |\Gamma|) \otimes \Delta(\sigma \upharpoonright |\Delta|)$ for all $\sigma : |\Gamma \otimes \Delta|$.
\end{definition}

\begin{definition}[countable tensor \cite{yamada2022game}]
\label{DefCountableTensor}
The \emph{countable tensor} of a family $(G_i)_{i \in \mathbb{N}}$ of subgames $G_i \subseteq H$ is the subgame $\otimes_{i \in \mathbb{N}}G_i \colonequals \{ \, \boldsymbol{s} \in \oc H \mid \forall j \in \mathbb{N} . \, \boldsymbol{s} \upharpoonright j \in G_j \, \} \subseteq \oc H$.
\end{definition} 

\begin{definition}[exponential on p-games \cite{yamada2022game}]
The \emph{exponential} of a p-game $\Gamma$ is the p-game $\oc \Gamma$ defined by $|\oc \Gamma| \colonequals \oc |\Gamma|$ and $(\oc \Gamma)(\sigma) \colonequals \otimes_{i \in \mathbb{N}} \Gamma(\sigma \upharpoonright i)$ for all $\sigma : | \oc \Gamma |$.
\end{definition}

Hence, strategies on $\Gamma \mathbin{\&} \Delta$ are the pairings $\langle \gamma, \delta \rangle$ of $\gamma : \Gamma$ and $\delta : \Delta$, strategies on $\Gamma \otimes \Delta$ are the tensors $\gamma \otimes \delta$ of $\gamma : \Gamma$ and $\delta : \Delta$, and strategies on $\oc \Gamma$ are those $\sigma : \oc |\Gamma|$ such that $\{ \, \boldsymbol{s} \upharpoonright i \mid \boldsymbol{s} \in \sigma \, \} : \Gamma$ for all $i \in \mathbb{N}$.

\begin{definition}[categories of p-games \cite{yamada2022game}]
\label{DefGameSemanticCategories}
The category $\mathbb{PG}_\oc$ consists of

\begin{itemize}

\item Well-opened p-games as objects;

\item Strategies on the implication $\Gamma \Rightarrow \Delta$ as morphisms $\Gamma \rightarrow \Delta$;

\item The composition $\psi \bullet \phi \colonequals \psi \circ \phi^\dagger : \Gamma \Rightarrow \Theta$ of strategies as the composition of morphisms $\phi : \Gamma \rightarrow \Delta$ and $\psi : \Delta \rightarrow \Theta$;

\item The dereliction $\mathrm{der}_{|\Gamma|} : \Gamma \Rightarrow \Gamma$ as the identity $\mathrm{id}_\Gamma$ on each object $\Gamma$.

\end{itemize}

The subcategory $\mathbb{LPG}_\oc$ (resp. $\mathbb{WPG}_\oc$) of $\mathbb{PG}_\oc$ consists of well-founded, well-opened p-games as objects, and winning (resp. winning, well-bracketed) strategies as morphisms.
\end{definition}

Because the underlying logic of MLTT is intuitionistic, the previous work \cite{yamada2022game} focuses on the category $\mathbb{WPG}_\oc$.
It establishes game semantics of MLTT by showing that the category $\mathbb{WPG}_\oc$ gives rise to abstract semantics of MLTT, called a \emph{category with families (CwF)} \cite{dybjer1996internal}:

\begin{definition}[CwFs \cite{dybjer1996internal,hofmann1997syntax}]
\label{DefCwFs}
A \emph{category with families (CwF)} is a tuple 
\begin{equation*}
\mathcal{C} = (\mathcal{C}, \mathrm{Ty}, \mathrm{Tm}, \_\{\_\}, T, \_.\_, \mathrm{p}, \mathrm{v}, \langle\_,\_\rangle_\_),
\end{equation*}
where
\begin{itemize}

\item $\mathcal{C}$ is a category with a terminal object $T \in \mathcal{C}$;

\item $\mathrm{Ty}$ assigns, to each object $\Gamma \in \mathcal{C}$, a set $\mathrm{Ty}(\Gamma)$ of \emph{types} in the \emph{context} $\Gamma$;

\item $\mathrm{Tm}$ assigns, to each pair $(\Gamma, A)$ of an object $\Gamma \in \mathcal{C}$ and a type $A \in \mathrm{Ty}(\Gamma)$, a set $\mathrm{Tm}(\Gamma, A)$ of \emph{terms} of type $A$ in the context $\Gamma$;

\item To each morphism $\phi : \Delta \to \Gamma$, $\_\{\_\}$ assigns a map $\_\{\phi\} : \mathrm{Ty}(\Gamma) \to \mathrm{Ty}(\Delta)$, called the \emph{substitution on types}, and a family $(\_\{\phi\}_A)_{A \in \mathrm{Ty}(\Gamma)}$ of maps $\_\{\phi\}_A : \mathrm{Tm}(\Gamma, A) \to \mathrm{Tm}(\Delta, A\{\phi\})$, called the \emph{substitutions on terms};


\item $\_ . \_$ assigns, to each pair $(\Gamma, A)$ of a context $\Gamma \in \mathcal{C}$ and a type $A \in \mathrm{Ty}(\Gamma)$, a context $\Gamma . A \in \mathcal{C}$, called the \emph{comprehension} of $A$;

\item $\mathrm{p}$ (resp. $\mathrm{v}$) associates each pair $(\Gamma, A)$ of a context $\Gamma \in \mathcal{C}$ and a type $A \in \mathrm{Ty}(\Gamma)$ with a morphism $\mathrm{p}_A : \Gamma . A \to \Gamma$ (resp. a term $\mathrm{v}_A \in \mathrm{Tm}(\Gamma . A, A\{\mathrm{p}_A\})$), called the \emph{first projection} on $A$ (resp. the \emph{second projection} on $A$);

\item $\langle \_, \_ \rangle_\_$ assigns, to each triple $(\phi, A, \Check{\alpha})$ of a morphism $\phi : \Delta \to \Gamma$, a type $A \in \mathrm{Ty}(\Gamma)$ and a term $ \Check{\alpha} \in \mathrm{Tm}(\Delta, A\{\phi\})$, a morphism $\langle \phi,  \Check{\alpha} \rangle_A : \Delta \to \Gamma . A$, called the \emph{extension} of $\phi$ by $ \Check{\alpha}$,

\end{itemize}
that satisfies, for any $\Theta \in \mathcal{C}$, $\varphi : \Theta \to \Delta$ and $\alpha \in \mathrm{Tm}(\Gamma, A)$, the equations
\begin{mathpar}
A \{ \mathrm{id}_\Gamma \} = A
\and
A \{ \phi \circ \varphi \} = A \{ \phi \} \{ \varphi \}
\and
\alpha \{ \mathrm{id}_\Gamma \}_A = \alpha
\and
\alpha \{ \phi \circ \varphi \}_A = \alpha \{ \phi \}_A \{ \varphi \}_{A\{\phi\}}
\and
\mathrm{p}_A \circ \langle \phi,  \Check{\alpha} \rangle_A = \phi
\and
\mathrm{v}_A \{ \langle \phi,  \Check{\alpha} \rangle_A \} =  \Check{\alpha}
\and
\langle \phi, \Check{\alpha} \rangle_A \circ \varphi = \langle \phi \circ \varphi, \Check{\alpha} \{ \varphi \}_{A\{\phi\}} \rangle_A
\and
\langle \mathrm{p}_A, \mathrm{v}_A \rangle_A = \mathrm{id}_{\Gamma . A}.
\end{mathpar}
\end{definition}

We sometimes write $\mathrm{Ty}_{\mathcal{C}}$, $\mathrm{Term}_{\mathcal{C}}$ and so on when we would like to emphasise the underlying CwF $\mathcal{C}$.
Roughly, judgements of MLTT are interpreted in a CwF $\mathcal{C}$ by
\begin{mathpar}
\mathsf{\vdash \Gamma \ ctx} \mapsto \llbracket \mathsf{\Gamma} \rrbracket \in \mathcal{C} \and
\mathsf{\Gamma \vdash A \ type} \mapsto \llbracket \mathrm{A} \rrbracket \in \mathrm{Ty}(\llbracket \mathsf{\Gamma} \rrbracket) \and
\mathsf{\Gamma \vdash a : A} \mapsto \llbracket \mathrm{A} \rrbracket \in \mathrm{Tm}(\llbracket \mathsf{\Gamma} \rrbracket, \llbracket \mathrm{A} \rrbracket) \and
\mathsf{\vdash \Gamma = \Delta \ ctx} \Rightarrow \llbracket \mathsf{\Gamma} \rrbracket = \llbracket \mathsf{\Delta} \rrbracket \and
\mathsf{\Gamma \vdash A = B \ type} \Rightarrow \llbracket \mathrm{A} \rrbracket = \llbracket \mathsf{B} \rrbracket \and
\mathsf{\Gamma \vdash a = a' : A} \Rightarrow \llbracket \mathrm{A} \rrbracket = \llbracket \mathsf{a'} \rrbracket,
\end{mathpar}
where $\llbracket \_ \rrbracket$ denotes the \emph{semantic map} or \emph{interpretation}.
See \cite{hofmann1997syntax} for the details.

In the following, we recall the additional structures on the category $\mathbb{WPG}_\oc$ that lift it to a CwF.
First, types in the CwF $\mathbb{WPG}_\oc$ are \emph{dependent p-games}:
\begin{definition}[dependent p-games \cite{yamada2022game}]
\label{DefDependentPredicateGames}
A \emph{linearly dependent predicate (p-) game} over a p-game $\Gamma$ is a pair $L = (|L|, \| L \|)$ of a game $|L|$ and a family $\| L \| = (L(\gamma_0))_{\gamma_0 \in \mathbb{WPG}_\oc(\Gamma)}$ of p-games $L(\gamma_0)$ such that $|L(\gamma_0)| = |L|$.
It is \emph{well-opened} (resp. \emph{well-founded}) if so is $|L|$.
The \emph{extension} of the family $\| L \|$ is the family $L^\star = (L^\star(\gamma))_{\gamma : \Gamma}$ of p-games $L^\star(\gamma)$ defined by 
\begin{equation}
L^\star(\gamma) \colonequals \begin{cases} L(\gamma) &\text{if $\gamma \in \mathbb{WPG}_\oc(\Gamma)$;} \\ \mathscr{P}(|L|) &\text{otherwise.} \end{cases}
\end{equation}

A \emph{dependent predicate (p-) game} over $\Gamma$ is a linearly dependent one over the exponential $\oc \Gamma$.
\end{definition}

\begin{notation}
We write $\mathscr{D}_\ell(\Gamma)$ (resp. $\mathscr{D}_\ell^{\mathrm{w}}(\Gamma)$) for the set of all linearly dependent p-games (resp. well-opened, well-founded ones) over $\Gamma$, and $\{ \Gamma' \}_\Gamma$ or $\{ \Gamma' \}$ for the \emph{constant} one at $\Gamma'$, i.e., $\{ \Gamma' \}_\Gamma \colonequals (\Gamma', \gamma : \Gamma \mapsto \Gamma')$.
Let $\mathscr{D}(\Gamma) \colonequals \mathscr{D}_\ell(\oc \Gamma)$ and $\mathscr{D}^{\mathrm{w}}(\Gamma) \colonequals \mathscr{D}^{\mathrm{w}}_\ell(\oc \Gamma)$.
We often write $\gamma_0^\dagger$ for an arbitrary element of $\mathbb{WPG}_\oc(\oc \Gamma)$, where $\gamma_0 \in \mathbb{WPG}_\oc(\Gamma)$, since elements of $\mathbb{WPG}_\oc(\oc \Gamma)$ are all \emph{innocent} and so promotions of elements of $\mathbb{WPG}_\oc(\Gamma)$.
\end{notation}

Next, terms in the CwF $\mathbb{WPG}_\oc$ are winning, well-bracketed strategies on the following p-games:
\begin{definition}[linear-pi and pi \cite{yamada2022game}]
\label{DefLinearPiSpaces}
Let $L$ be a linearly dependent p-game over a p-game $\Gamma$, and $A$ be a dependent p-game over $\Gamma$.
The \emph{linear-pi} from $\Gamma$ to $L$ is the p-game $\Pi_\ell (\Gamma, L)$ defined by $|\Pi_\ell (\Gamma, L)| \colonequals |L|^{|\Gamma|}$ and for all $\phi : |\Pi_\ell (\Gamma, L)|$
\begin{align*}
\Pi_\ell (\Gamma, L)(\phi) &\colonequals \{ \boldsymbol{\epsilon} \} \cup \{ \, \boldsymbol{s}m \in |\Pi_\ell (\Gamma, L)|^{\mathrm{Odd}} \mid \boldsymbol{s} \in \Pi_\ell (\Gamma, L)(\phi), \exists \gamma : \Gamma . \, \boldsymbol{s}m \in L^\star(\gamma)(\phi \circ \gamma)^{\overline{\gamma}_\Gamma} \, \} \\
&\cup \{ \, \boldsymbol{t}lr \in |\Pi_\ell (\Gamma, L)|^{\mathrm{Even}} \mid \boldsymbol{t}l \in \Pi_\ell (\Gamma, L)(\phi), \forall \gamma : \Gamma . \, \boldsymbol{t}l \in L^\star(\gamma)(\phi \circ \gamma)^{\overline{\gamma}_\Gamma} \Rightarrow \boldsymbol{t}lr \in L^\star(\gamma)(\phi \circ \gamma)^{\overline{\gamma}_\Gamma} \, \},
\end{align*}
and the \emph{pi} from $\Gamma$ to $A$ is the linear-pi $\Pi (\Gamma, A) \colonequals \Pi_\ell(\oc \Gamma, A)$. 
We write $\Gamma \Rightarrow A$ for $\Pi (\Gamma, A)$ if $A$ is constant. 
\end{definition}

Finally, comprehensions in the CwF $\mathbb{WPG}_\oc$ are given by:
\begin{definition}[sigma \cite{yamada2022game}]
\label{DefSigma}
The \emph{sigma} of a p-game $\Gamma$ and a dependent p-game $A$ over $\Gamma$ is the p-game $\Sigma (\Gamma, A)$ defined by $|\Sigma (\Gamma, A)| \colonequals |\Gamma| \mathbin{\&} |A|$ and $\Sigma (\Gamma, A)(\langle \gamma, \alpha \rangle) \colonequals \Gamma(\gamma) \mathbin{\&} A^\star(\gamma^\dagger)(\alpha)$ for all $\langle \gamma, \alpha \rangle : |\Sigma(\Gamma, A)|$. 
We write $\Gamma \mathbin{\&} A$ for $\Sigma(\Gamma, A)$ if $A$ is constant. 
\end{definition}

We are now ready to recall:
\begin{theorem}[a game-semantic CwF \cite{yamada2022game}]
\label{DefCwFWPG}
The category $\mathbb{WPG}_\oc$ gives rise to a CwF as follows:
\begin{itemize}

\item The terminal p-game $T \in \mathbb{WPG}_\oc$ in Example~\ref{ExPGames} forms a terminal object;

\item We define $\mathrm{Ty}(\Gamma) \colonequals \mathscr{D}^{\mathrm{w}}(\Gamma)$ ($\Gamma \in \mathbb{WPG}_\oc$) and $\mathrm{Tm}(\Gamma, A) \colonequals \mathbb{WPG}_\oc(\Pi(\Gamma, A))$ ($A \in \mathscr{D}^{\mathrm{w}}(\Gamma)$); 

\item Given a morphism $\phi : \Delta \rightarrow \Gamma$, we define $\_\{\phi\} : \mathrm{Ty}(\Gamma) \to \mathrm{Ty}(\Delta)$ by $|A\{ \phi \}| \colonequals |A|$ and $A \{ \phi \}(\delta_0^\dagger) \colonequals A(\phi^\dagger \bullet \delta_0)$ for all $A \in \mathrm{Ty}(\Gamma)$ and $\delta_0^\dagger \in \mathbb{WPG}_\oc(\oc \Delta)$, and define $\_\{\phi\}_A : \mathrm{Tm}(\Gamma, A) \to \mathrm{Tm}(\Delta, A\{\phi\})$ by $\alpha \{ \phi \}_A \colonequals \alpha \bullet \phi$ for all $\alpha \in \mathrm{Tm}(\Gamma, A)$; 


\item We define $\Gamma.A \colonequals \Sigma (\Gamma, A)$, $\mathrm{p}_A \colonequals \mathrm{der}_{|\Gamma|} : \Sigma (\Gamma, A) \rightarrow \Gamma$, $\mathrm{v}_A \colonequals \mathrm{der}_{|A|} : \Pi (\Sigma(\Gamma, A), A\{\mathrm{p}_A\})$ and $\langle \phi, \Check{\alpha} \rangle_A \colonequals  \langle \phi, \Check{\alpha} \rangle : \Delta \rightarrow \Sigma (\Gamma, A)$ ($\Check{\alpha} \in \mathrm{Tm}(\Delta, A\{ \phi \})$).


\end{itemize}

\end{theorem}

Given $\Gamma \in \mathbb{WPG}_\oc$ and $A \in \mathscr{D}^{\mathrm{w}}(\Gamma)$, we write $\mathbb{WPG}_\oc(\Gamma, A)$ for the set $\mathrm{Tm}(\Gamma, A)$ of all terms.
We often omit subscripts on components of $\mathbb{WPG}_\oc$ when they are evident. 

Strictly speaking, a CwF only interprets the core part of MLTT common to all types.
For interpreting One-, Zero-, N-, Pi-, Sigma- and Id-types, we need to equip the CwF $\mathbb{WPG}_\oc$ with \emph{semantic type formers} \cite{hofmann1997syntax} that interpret these types. 
In the following, we only sketch the game-semantic type formers on the CwF $\mathbb{WPG}_\oc$, leaving the general definition of semantic type formers to Hofmann \cite{hofmann1997syntax}.
Let us fix an objects $\Delta, \Gamma \in \mathbb{WPG}_\oc$ and types $A \in \mathscr{D}^{\mathrm{w}}(\Gamma)$ and $B \in \mathscr{D}^{\mathrm{w}}(\Sigma(\Gamma, A))$. 

\begin{theorem}[game semantics of Pi-types \cite{yamada2022game}]
\label{ThmGameSemanticsOfPiType}
$\mathbb{WPG}_\oc$ strictly supports Pi-types, where
\begin{itemize}

\item \textsc{($\Pi$-Form)} A dependent p-game $\Pi(A, B) \in \mathscr{D}^{\mathrm{w}}(\Gamma)$ is given by $|\Pi(A, B)| \colonequals |A| \Rightarrow |B|$ and $\Pi(A, B)(\gamma_0^\dagger) \colonequals \Pi(A(\gamma_0^\dagger), B_{\gamma_0^\dagger})$ for each $\gamma_0^\dagger \in \mathbb{WPG}_\oc(\oc \Gamma)$, and another dependent p-game $B_{\gamma_0^\dagger} \in \mathscr{D}^{\mathrm{w}}(A(\gamma_0^\dagger))$ by $|B_{\gamma_0^\dagger}| \colonequals |B|$ and $B_{\gamma_0^\dagger}(\alpha_0^\dagger) \colonequals B(\langle \gamma_0, \alpha_0 \rangle^\dagger)$ for each $\alpha_0^\dagger \in \mathbb{WPG}_\oc(\oc A(\gamma_0^\dagger))$.
We write $A \Rightarrow B$ for $\Pi(A, B)$ if $B_{\gamma_0^\dagger}$ is constant for each $\gamma_0^\dagger \in \mathbb{WPG}_\oc(\oc \Gamma)$.
Note that the equation 
\begin{equation}
\label{PiSubst}
\Pi(A, B)\{ \phi \} = \Pi(A\{ \phi \}, B\{ \phi_A^+ \})
\end{equation}
holds for each morphism $\phi : \Delta \rightarrow \Gamma$ (\textsc{$\Pi$-Subst}), where $\phi_A^+ \colonequals \langle \phi \bullet \mathrm{p}, \mathrm{v} \rangle : \Delta . A\{ \phi \} \rightarrow \Gamma . A$.
 
\item \textsc{($\Pi$-Intro)} Given a term $\beta \in \mathbb{WPG}_\oc(\Sigma(\Gamma, A), B)$, another term $\lambda_{A, B}(\beta) \in \mathbb{WPG}_\oc(\Gamma, \Pi(A, B))$ is obtained from $\beta$ by adjusting tags or \emph{currying} $\beta$ with respect to the adjunction between tensor $\otimes$ and linear implication $\multimap$ \cite{mccusker1998games} (thanks to the evident isomorphism $|\oc \Sigma(\Gamma, A)| = \oc (|\Gamma| \mathbin{\&} |A|) \cong \oc |\Gamma| \otimes \oc |A|$).
We often omit the subscripts $(\_)_{A, B}$ on $\lambda_{A, B}$ and the inverse $\lambda_{A, B}^{-1}$.

\item \textsc{($\Pi$-Elim)} We define $\mathrm{App}_{A, B} (\kappa, \alpha) \colonequals \lambda_{A, B}^{-1}(\kappa) \{\overline{\alpha} \} \in \mathbb{WPG}_\oc(\Gamma, B\{ \overline{\alpha} \})$ for all $\kappa \in \mathbb{WPG}_\oc(\Gamma, \Pi(A, B))$ and $\alpha \in \mathbb{WPG}_\oc(\Gamma, A)$. 
We often omit the subscripts $(\_)_{A, B}$ on $\mathrm{App}_{A, B}$.

\end{itemize}

\end{theorem}

\begin{theorem}[game semantics of Sigma-types \cite{yamada2022game}]
\label{ThmGameSemanticsOfSigmaType}
$\mathbb{WPG}_\oc$ strictly supports Sigma-types, where

\begin{itemize}

\item \textsc{($\Sigma$-Form)} Similarly to Pi-types, we define $\Sigma (A, B) \colonequals (|A| \mathbin{\&} |B|, (\Sigma(A(\gamma_0^\dagger), B_{\gamma_0^\dagger}))_{\gamma_0^\dagger \in \mathbb{WPG}_\oc(\oc \Gamma)})$.
We write $A \mathbin{\&} B$ for $\Sigma(A, B)$ if $B_{\gamma_0^\dagger}$ is constant for each $\gamma_0^\dagger \in \mathbb{WPG}_\oc(\oc \Gamma)$.

\item \textsc{($\Sigma$-Intro)} By the evident bijection $\Sigma(\Sigma(\Gamma,A),B) \cong \Sigma(\Gamma, \Sigma(A,B))$, we define a morphism $\mathrm{Pair}_{A, B} \colonequals \langle \mathrm{p}_A \bullet \mathrm{p}_B, \langle \mathrm{v}_A \{ \mathrm{p}_B \}, \mathrm{v}_B \rangle \rangle : \Sigma(\Sigma(\Gamma,A),B) \stackrel{\sim}{\rightarrow} \Sigma(\Gamma, \Sigma(A,B))$. 

\item \textsc{($\Sigma$-Elim)} Given a term $\rho \in \mathbb{WPG}_\oc(\Sigma(\Sigma(\Gamma, A), B), P\{\mathrm{Pair}_{A,B}\})$, we define another term $\mathcal{R}^{\Sigma}_{A, B, P}(\rho) \colonequals \rho \{ \mathrm{Pair}_{A, B}^{-1} \} \in \mathbb{WPG}_\oc(\Sigma(\Gamma, \Sigma (A, B)), P\{\mathrm{Pair}_{A,B}\}\{\mathrm{Pair}_{A,B}^{-1}\}) = \mathbb{WPG}_\oc(\Sigma(\Gamma, \Sigma (A, B)), P)$.

\end{itemize}
\end{theorem}

\begin{theorem}[game semantics of atomic types \cite{yamada2022game}]
\label{ThmGameSemanticsOfSigmaType}
$\mathbb{WPG}_\oc$ supports One-, Zero- and N-types, where their formation rules are given by constant dependent p-gams at the terminal p-game $T$, the empty p-game $0$ and the natural number p-game $N$, for which we write $1$, $0$ and $N$, respectively
\end{theorem}

\begin{theorem}[game semantics of Id-types \cite{yamada2022game}]
\label{LemGameSemanticIdTypes}
$\mathbb{WPG}_\oc$ supports Id-types, where
\begin{itemize}

\item \textsc{(Id-Form)} Let $T' \colonequals \mathrm{flat}(\{ \mathbin{\surd} \})$ (Example~\ref{ExamplesOfGames}), where $\mathbin{\surd}$ is any element.
We define a dependent p-game $\mathrm{Id}_A \in \mathscr{D}^{\mathrm{w}}(\Sigma(\Sigma(\Gamma, A), A^+))$ by $|\mathrm{Id}_A| \colonequals T'$ and $\mathrm{Id}_A(\langle \langle \gamma_0, \alpha_0 \rangle, \alpha_0' \rangle^\dagger) \colonequals \begin{cases} (T', \kappa_{T'}) &\text{if $\alpha_0 = \alpha_0'$;} \\ (T', \kappa_{0}) &\text{otherwise,} \end{cases}$ for all $\langle \langle \gamma_0, \alpha_0 \rangle, \alpha_0' \rangle^\dagger \in \mathbb{WPG}_\oc(\oc \Sigma(\Sigma(\Gamma, A), A^+))$, where $\kappa_{X}$ is the constant family at a game $X$.

\item \textsc{(Id-Intro)} Let $\mathrm{Refl}_A \colonequals \langle \overline{\mathrm{v}_A}, \mathrm{refl}_A \rangle \in \mathbb{WPG}_\oc(\Sigma(\Gamma, A), \Sigma(\Sigma(\Sigma(\Gamma, A), A^+), \mathrm{Id}_A))$, where $\mathrm{refl}_A \in \mathbb{WPG}_\oc(\Sigma(\Gamma, A), \mathrm{Id}_A\{ \overline{\mathrm{v}_A} \})$ is $\underline{\surd} : T'$ (Example~\ref{ExamplesOfGames}) up to tags.


\end{itemize}

\end{theorem}

\begin{example}
\label{ExCrucialExample}
Consider the interpretation of the Id-type $\mathsf{f : N \Rightarrow N, g : N \Rightarrow N \vdash Id_{N \Rightarrow N}(f, g) \ type}$ in the CwF $\mathbb{WPG}_\oc$ \cite{yamada2022game}, which is the p-game $\Pi((N \Rightarrow N) \mathbin{\&} (N \Rightarrow N), \mathrm{Id}_{N \Rightarrow N}(\pi_1, \pi_2))$.
The component of the codomain $\mathrm{Id}_{N \Rightarrow N}(\pi_1, \pi_2)$ is in general not decidable because any play in this p-game can observe only \emph{finite} information about two input strategies on the domain.
Nevertheless, this is not a problem because the codomain component of each pi (Definition~\ref{DefLinearPiSpaces}) is specified \emph{only gradually} (and often incompletely) along the gradual (and often incomplete) disclosure of input strategies on the domain by Opponent.

Accordingly, assuming a p-game $\mathcal{U}$ for the universe, a strategy $\mathrm{En}(\mathrm{Id}_{N \Rightarrow N}(\pi_1, \pi_2)) : (N \Rightarrow N) \mathbin{\&} (N \Rightarrow N) \rightarrow \mathcal{U}$ that encodes the p-game $\Pi((N \Rightarrow N) \mathbin{\&} (N \Rightarrow N), \mathrm{Id}_{N \Rightarrow N}(\pi_1, \pi_2))$, if any, only has to encode the currently possible components of the codomain $\mathrm{Id}_{N \Rightarrow N}(\pi_1, \pi_2)$ at each moment; it does not have to decide if the two input strategies on the domain are equal. 
We emphasise that this \emph{intensionality} is highly nontrivial, and it distinguishes game semantics from other semantics of MLTT such as domains and realisability \cite{palmgren1993information,streicher2012semantics,blot2018extensional}.
Moreover, this observation is the starting point of our solution to the main problem (\S\ref{GameSemanticsOfUniverses}) in achieving game semantics of the universe sketched in \S\ref{Solution}. 
\end{example}

\section{Game semantics of universes}
\label{GameSemanticsOfUniverses}
This section presents our main contribution: game semantics of the cumulative hierarchy of universes. 

To this end, we first recall the semantic type former for the cumulative hierarchy of universes:

\begin{definition}[categorical semantics of universes \cite{hofmann1997syntax}]
\label{DefSemanticTypeFormerForUniverses}
A CwF $\mathcal{C}$ \emph{supports universes} if
\begin{itemize}

\item \textsc{(U-Form)} Given an object $\Gamma \in \mathcal{C}$, there is a type $\mathcal{U}_k^{[\Gamma]} \in \mathrm{Ty}(\Gamma)$ for each natural number $k \in \mathbb{N}$, called the \emph{$(k+1)$st universe} in the context $\Gamma$, where we often omit the superscript $(\_)^{[\Gamma]}$ (when the object $\Gamma$ is obvious) and/or the subscript $(\_)_k$ (when the index $k$ is unimportant);

\item \textsc{(U-Intro)} Given a type $A \in \mathrm{Ty}(\Gamma)$, there is a term $\mathrm{En}_k(A) \in \mathrm{Tm}(\Gamma, \mathcal{U}_k)$ for some $k \in \mathbb{N}$, subsuming $\mathrm{En}_k(\mathcal{U}_k^{[\Gamma]}) \in \mathrm{Tm}(\Gamma, \mathcal{U}^{[\Gamma]}_{k+1})$ for each  $k \in \mathbb{N}$, where we often omit the subscript $(\_)_k$ on $\mathrm{En}$;

\item \textsc{(U-Elim)} Each term $\psi \in \mathrm{Tm}(\Gamma, \mathcal{U}_k)$ induces a type $\mathrm{El}_k(\psi) \in \mathrm{Ty}(\Gamma)$, where we often omit the subscript $(\_)_k$ on $\mathrm{El}$;

\item \textsc{(U-Comp)} $\mathrm{El}(\mathrm{En}(A)) = A$;

\item \textsc{(U-Cumul)} If $\psi \in \mathrm{Tm}(\Gamma, \mathcal{U}_k)$, then $\psi \in \mathrm{Tm}(\Gamma, \mathcal{U}_{k+1})$;

\item \textsc{(U-Subst)} $\mathcal{U}_k^{[\Gamma]}\{ \phi \} = \mathcal{U}_k^{[\Delta]} \in \mathrm{Ty}(\Delta)$ for each morphism $\phi : \Delta \to \Gamma$;

\item \textsc{($\mathrm{En}$-Subst)} $\mathrm{En}(A)\{ \phi \} = \mathrm{En}(A \{ \phi \}) \in \mathrm{Tm}(\Delta, \mathcal{U})$.
\end{itemize}

\end{definition}

The axiom \textsc{U-Cumul} requires the hierarchy $(\mathcal{U}_k)_{k \in \mathbb{N}}$ of universes $\mathcal{U}_k$ to be \emph{cumulative}.
For achieving game semantics of the cumulative hierarchy of universes, it suffices to equip our game-semantic CwF $\mathbb{WPG}_\oc$ with this semantic type former because then the semantic type former will automatically induce game semantics of the cumulative hierarchy of universes as described in Hofmann \cite{hofmann1997syntax}.

\subsection{Universe predicate games}
\label{UniversePredicativeGames}
For convenience, we employ the following reformulation $\mathrm{Id}'(\alpha, \alpha')$ of Id-types that satisfy the axiom \textsc{Id'-Subst} corresponding to \textsc{Id-Subst}.
In fact, the type $\mathrm{Id}_A \in \mathrm{Ty}(\Gamma . A . A^+)$ is equivalent to the family $(\mathrm{Id}'_A(\alpha, \alpha'))_{\alpha, \alpha' \in \mathrm{Tm}(\Gamma, A)}$ of types $\mathrm{Id}'_A(\alpha, \alpha') \in \mathrm{Ty}(\Gamma)$: The former is recovered from the latter by
\begin{equation*}
\mathrm{Id}_A \colonequals \mathrm{Id}'_A(\mathrm{v}\{ \mathrm{p} \}, \mathrm{v}),
\end{equation*}
and the latter from the former by
\begin{equation*}
\mathrm{Id}'_A(\alpha, \alpha') \colonequals \mathrm{Id}_A\{ \langle \langle \mathrm{id}_\Gamma, \alpha \rangle, \alpha' \rangle \}.
\end{equation*}
We also note that the axiom \textsc{Id-Subst} implies the equation
\begin{align*}
\mathrm{Id}'_A(\alpha, \alpha')\{ \phi  \} &= \mathrm{Id}_A\{ \langle \langle \mathrm{id}_\Gamma, \alpha \rangle, \alpha' \rangle \}\{ \phi \} \\
&= \mathrm{Id}_A\{ \langle \langle \phi, \alpha\{ \phi \} \rangle, \alpha'\{ \phi \} \rangle \} \\
&= \mathrm{Id}_A\{ \langle \langle \phi \circ \mathrm{p}, \mathrm{v} \rangle \circ \mathrm{p}, \mathrm{v} \rangle \}\{ \langle \langle \mathrm{id}_\Gamma, \alpha\{ \phi \} \rangle, \alpha'\{ \phi \} \rangle \} \\
&= \mathrm{Id}_A\{ \phi_{A, A^+}^{++} \}\{ \langle \langle \mathrm{id}_\Gamma, \alpha\{ \phi \} \rangle, \alpha'\{ \phi \} \rangle \} \\
&= \mathrm{Id}_{A\{ \phi \}}\{ \langle \langle \mathrm{id}_\Gamma, \alpha\{ \phi \} \rangle, \alpha'\{ \phi \} \rangle \} \quad \text{(by \textsc{Id-Subst})} \\
&= \mathrm{Id}'_{A\{ \phi \}}(\alpha\{ \phi \}, \alpha'\{ \phi \})
\end{align*}
for each terms $\alpha, \alpha' \in \mathrm{Tm}(\Gamma, A)$ and morphism $\phi : \Delta \rightarrow \Gamma$, which we call the axiom \textsc{Id'-Subst}.
Conversely, the axiom \textsc{Id'-Subst} implies the axiom \textsc{Id-Subst} because
\begin{align*}
\mathrm{Id}_A\{ \phi_{A, A^+}^{++}  \} &= \mathrm{Id}'_A(\mathrm{v}\{ \mathrm{p} \}, \mathrm{v})\{  \langle \phi_A^+ \circ \mathrm{p}, \mathrm{v} \rangle \} \\
&= \mathrm{Id}'_{A\{ \phi \}}(\mathrm{v}\{ \mathrm{p} \}\{ \langle \phi_A^+ \circ \mathrm{p}, \mathrm{v} \rangle \}, \mathrm{v}\{ \langle \phi_A^+ \circ \mathrm{p}, \mathrm{v} \rangle \}) \quad \text{(by \textsc{Id'-Subst})} \\
&= \mathrm{Id}'_{A\{ \phi \}}(\mathrm{v}\{ \langle \phi \circ \mathrm{p} \circ \mathrm{p}, \mathrm{v}\{ \mathrm{p} \} \rangle \}, \mathrm{v}) \\
&= \mathrm{Id}'_{A\{ \phi \}}(\mathrm{v}\{ \mathrm{p} \}, \mathrm{v}) \\
&= \mathrm{Id}_{A\{ \phi \}}.
\end{align*}
This in particular implies the equation
\begin{equation}
\label{IdSubstBeta}
\mathrm{Id}'_A(\alpha, \alpha')(\gamma_0^\dagger) = \mathrm{Id}'_{A(\gamma_0^\dagger)}(\alpha \bullet \gamma_0, \alpha' \bullet \gamma_0)
\end{equation}
for all $\gamma_0^\dagger \in \mathbb{WPG}_\oc$.
We leave it to the reader to reformulate the other axioms on Id-types in such a way that they correspond to this reformulation.
From now on, we simply write $\mathrm{Id}(\alpha, \alpha')$ for $\mathrm{Id}'(\alpha, \alpha')$.

Then, as sketched in \S\ref{Solution}, the main idea for the construction of our game-semantic type former for the cumulative hierarchy of universes is centred around the following \emph{universe p-games}:
\begin{definition}[universe p-games]
\label{DefUniversePredicativeGames}
Let us fix an injection $\sharp_0 : \{ 1, 0, N, \Pi, \Sigma, \mathrm{Id} \} \rightarrowtail \mathbb{N}$.
For each natural number $k \in \mathbb{N}$, the \emph{$(k+1)$st universe predicate (p-) game} is the constant p-game $\mathcal{U}_k$ on the game $|\mathcal{U}_k| \colonequals \bigcup_{i \in \mathbb{N}} |\mathcal{U}_k^{(i)}|$ together with an arbitrarily fixed injection $\sharp_k : \{ \, 1, 0, N, \Pi, \Sigma, \mathrm{Id} \, \} \uplus \{ \, \mathcal{U}_j \mid j < k \, \} \rightarrowtail \mathbb{N}$ that conservatively extends $\sharp_{k-1}$, where $\mathcal{U}_k^{(i)}$ is a p-game inductively defined as follows:
\begin{enumerate}

\item \textsc{(Base case)} We define the p-game 
\begin{equation*}
\mathcal{U}_k^{(0)} \colonequals \mathscr{P}(\mathrm{Pref}(\{ \, q^{\mathrm{OQ}} . \sharp_k(X)^{\mathrm{PA}} \mid X \in \{ 1, 0, N, \mathcal{U}_j \}, j < k \, \})),
\end{equation*}
where $q^{\mathrm{OQ}}$ justifies $\sharp_k(X)^{\mathrm{PA}}$, together with a function
\begin{align*}
\mathrm{El}^{(0)}_k : \mathbb{WPG}_\oc(\oc \mathcal{U}_k^{(0)}) &\rightarrow \mathrm{ob}(\mathbb{WPG}_\oc) \\
\underline{\sharp_k(X)}^\dagger &\mapsto X.
\end{align*}
Abusing notation, we lift this function to a dependent p-game $\mathrm{El}^{(0)}_k \in \mathscr{D}(\mathcal{U}_k^{(0)})$ by
\begin{mathpar}
|\mathrm{El}^{(0)}_k| \colonequals \bigcup_{\underline{\sharp_k(X)}^\dagger \in \mathbb{WPG}_\oc(\oc \mathcal{U}_k^{(0)})} \mathrm{El}^{(0)}_k(\underline{\sharp_k(X)}^\dagger)
\and
\|\mathrm{El}^{(0)}_k\| : \underline{\sharp_k(X)}^\dagger \mapsto \mathrm{El}^{(0)}_k(\underline{\sharp_k(X)}^\dagger),
\end{mathpar}
where recall Example~\ref{ExamplesOfStrategies} for the notation $\underline{\sharp_k(X)}$.
We also write $\mathcal{U}_k^{(0)}$ for the constant p-game $\{ \mathcal{U}_k^{(0)} \}$.

Moreover, for each object $\Gamma \in \mathbb{WPG}_\oc$, we further lift this dependent p-game to a function
\begin{align*}
\mathrm{El}^{(0)}_{k, \Gamma} : \mathbb{WPG}_\oc(\Gamma, \mathcal{U}_k^{(0)}) &\rightarrow \mathscr{D}(\Gamma) \\
\psi &\mapsto \mathrm{El}^{(0)}_{k, \Gamma}(\psi),
\end{align*}
where the dependent p-game $\mathrm{El}^{(0)}_{k, \Gamma}(\psi) \in \mathscr{D}(\Gamma)$ is given by
\begin{mathpar}
|\mathrm{El}^{(0)}_{k, \Gamma}(\psi)| \colonequals \bigcup_{\gamma_0^\dagger \in \mathbb{WPG}_\oc(\oc \Gamma)} \mathrm{El}^{(0)}_k(\psi \bullet \gamma_0)
\and
\|\mathrm{El}^{(0)}_{k, \Gamma}(\psi)\| : \gamma_0^\dagger \mapsto \mathrm{El}^{(0)}_k(\psi \bullet \gamma_0).
\end{mathpar}
This function $\mathrm{El}^{(0)}_{k, \Gamma}$ generalises the dependent p-game $\mathrm{El}^{(0)}_k$ due to the evident isomorphism $\mathrm{El}^{(0)}_{k, T} \cong \mathrm{El}^{(0)}_k$.
We usually omit the subscript $(\_)_\Gamma$ on the function $\mathrm{El}^{(0)}_{k, \Gamma}$ when it does not bring confusion. 

\item \textsc{(Inductive step)} We define the p-game 
\begin{align*}
\mathcal{U}_k^{(i+1)} \colonequals \mathscr{P}(\mathcal{U}_k^{(i)} &\cup \mathrm{Pref}(\{ \, q^{\mathrm{OQ}} . \sharp_k(Y)^{\mathrm{PA}} . \boldsymbol{s} \mid Y \in \{ \Pi, \Sigma \}, \boldsymbol{s} \in \Sigma(\mathcal{U}_k^{(i)}, \mathrm{El}^{(i)}_k \Rightarrow \mathcal{U}_k^{(i)}) \, \}) \\
&\cup  \mathrm{Pref}(\{ \, q^{\mathrm{OQ}} . \sharp_k(\mathrm{Id})^{\mathrm{PA}} . \boldsymbol{t} \mid \boldsymbol{t} \in \Sigma(\mathcal{U}_k^{(i)}, \mathrm{El}^{(i)}_k \mathbin{\&} \mathrm{El}^{(i)}_k) \, \})),
\end{align*}
where $q^{\mathrm{OQ}}$ justifies both $\sharp(Y)^{\mathrm{PA}}$ and $\sharp(\mathrm{Id})^{\mathrm{PA}}$, and in turn the latter two moves justify the initial moves in $\boldsymbol{s}$ and $\boldsymbol{t}$, respectively, together with a function
\begin{align*}
\mathrm{El}^{(i+1)}_k : \mathbb{WPG}_\oc(\oc \mathcal{U}_k^{(i+1)}) &\rightarrow \mathrm{ob}(\mathbb{WPG}_\oc) \\
\underline{\sharp_k(X)}^\dagger &\mapsto X \\
q . \sharp_k(Y) . \langle \mu, \psi \rangle^\dagger &\mapsto Y(\mathrm{El}^{(i)}_k(\mu), \mathrm{El}^{(i)}_k(\psi)) \\
q . \sharp_k(\mathrm{Id}) . \langle \mu, \langle \alpha, \alpha' \rangle \rangle^\dagger &\mapsto \mathrm{Id}_{\mathrm{El}^{(i)}_k(\mu)}(\mathrm{El}^{(i)}_k(\alpha), \mathrm{El}^{(i)}_k(\alpha')),
\end{align*}
where $q . a . \sigma \colonequals \mathrm{Pref}(\{ \, q^{\mathrm{OQ}}a^{\mathrm{PA}} \boldsymbol{v} \mid \boldsymbol{v} \in \sigma \, \})^{\mathrm{Even}}$ for each question $q^{\mathrm{OQ}}$, answer $a^{\mathrm{PA}}$ and strategy $\sigma$, and $q^{\mathrm{OQ}}$ justifies $a^{\mathrm{PA}}$, and $a^{\mathrm{PA}}$ justifies initial moves occurring in $\boldsymbol{v}$.

Again, we lift this function $\mathrm{El}^{(i+1)}_k$ to a dependent p-game $\mathrm{El}^{(i+1)}_k \in \mathscr{D}(\mathcal{U}_k^{(i+1)})$ and further to a function $\mathbb{WPG}_\oc(\Gamma, \mathcal{U}_k^{(i+1)}) \rightarrow \mathscr{D}(\Gamma)$ for each $\Gamma \in \mathbb{WPG}_\oc$ in the same way as the case of $\mathrm{El}^{(0)}_k$.
We also write $\mathcal{U}_k^{(i+1)}$ for the constant p-game $\{ \mathcal{U}_k^{(i+1)} \}$ and apply the notations for $\mathrm{El}^{(0)}_k$ to $\mathrm{El}^{(i+1)}_k$.

\end{enumerate}
Given an object $\Gamma \in \mathbb{WPG}_\oc$, we write $\mathcal{U}_k^{[\Gamma]} \in \mathscr{D}(\Gamma)$ for the constant dependent p-game at $\mathcal{U}_k$ and we often omit the superscript $(\_)^{[\Gamma]}$ on $\mathcal{U}_k^{[\Gamma]}$ when it does not bring confusion.

We finally define the injection
\begin{equation*}
\sharp \colonequals \bigcup_{k \in \mathbb{N}}\sharp_k : \{ \, 1, 0, N, \Pi, \Sigma, \mathrm{Id} \, \} \uplus \{ \, \mathcal{U}_j \mid j \in \mathbb{N} \, \} \rightarrowtail \mathbb{N}.
\end{equation*}
\end{definition}

Let us emphasise that the inductive step in Definition~\ref{DefUniversePredicativeGames} properly implements our idea on how to encode game semantics of Pi-, Sigma- and Id-types by strategies on games (\S\ref{Solution}) by nontrivial recursion. 
Specifically, our key technique is to define each universe p-game $\mathcal{U}_k$ inductively in terms of the games $\mathcal{U}_k^{(i)}$ ($i \in \mathbb{N}$) along with the construction of the function $\mathrm{El}_k^{(i)}$.
This is the highlight of the present work.

\subsection{Computational game semantics of the cumulative hierarchy of universes}
\label{ComputationalGameSemanticsOfCumulativeHierarchyOfUniverses}
We need one more preparation for our game semantics of universes as follows.
The axiom \textsc{U-Intro} (Definition~\ref{DefSemanticTypeFormerForUniverses}) requires that every type $A$ has its encoding $\mathrm{El}(A)$.
As already indicated in \S\ref{Solution}, however, we define the encoding function $\mathrm{En}$ inductively along the construction of types.
Accordingly, we have to restrict types in the CwF $\mathbb{WPG}_\oc$ to those freely generated by the type constructions, leading to:
\begin{definition}[a subCwF $\mathbb{UPG}_\oc$]
Let $\mathbb{UPG}_\oc \hookrightarrow \mathbb{WPG}_\oc$ be the substructural CwF of $\mathbb{WPG}_\oc$ such that
\begin{itemize}


\item The underlying category $\mathbb{UPG}_\oc$ is the category $\mathbb{WPG}_\oc$;

\item The types of $\mathbb{UPG}_\oc$ are inductively constructed from the atomic dependent p-games $1$, $0$, $N$ and $\mathcal{U}_k$ for all $k \in \mathbb{N}$ by the constructions $\Pi$, $\Sigma$ and $\mathrm{Id}$;

\item The terms of $\mathbb{UPG}_\oc$ are given by $\mathrm{Tm}_{\mathbb{UPG}_\oc}(\Gamma, A) \colonequals \mathrm{Tm}_{\mathbb{WPG}_\oc}(\Gamma, A)$ for all $\Gamma \in \mathbb{UPG}_\oc$ and $A \in \mathrm{Ty}_{\mathbb{UPG}_\oc}(\Gamma)$.

\end{itemize}
\end{definition}

\begin{corollary}[well-defined $\mathbb{UPG}_\oc$]
The structure $\mathbb{UPG}_\oc$ forms a well-defined CwF that supports One-, Zero-, N-, Pi-, Sigma- and Id-types in the same way as the CwF $\mathbb{WPG}_\oc$.
\end{corollary}
\begin{proof}
This corollary immediately follows from Theorem~\ref{DefCwFWPG} (where the only nontrivial point is the closure of types $\mathbb{UPG}_\oc$ under substitution, but it is easily shown by induction on the types).
\end{proof}

In addition, this CwF $\mathbb{UPG}_\oc$ also supports the cumulative hierarchy of universes: 
\begin{theorem}[game semantics of universes]
\label{ThmGameSemanticsOfUniverses}
The CwF $\mathbb{UPG}_\oc$ supports universes. 
\end{theorem}
\begin{proof}
Let $\Delta, \Gamma \in \mathbb{UPG}_\oc$, $A \in \mathrm{Ty}_{\mathbb{UPG}_\oc}(\Gamma)$ and $\phi \in \mathbb{UPG}_\oc(\Delta, \Gamma)$.

\begin{itemize}

\item \textsc{(U-Form)} We have $\mathcal{U}_k^{[\Gamma]} \in \mathscr{D}(\Gamma)$ for each natural number $k \in \mathbb{N}$ (Definition~\ref{DefUniversePredicativeGames}).

\item \textsc{(U-Intro)} Because $A$ is constructed inductively, we can define a term $\mathrm{En}(A) \in \mathrm{Tm}_{\mathbb{UPG}_\oc}(\Gamma, \mathcal{U}_{k(A)})$ for some natural number $k(A) \in \mathbb{N}$ inductively along the construction of $A$ as follows:
\begin{enumerate}

\item If $A$ is $1$, $0$ or $N$, then
\begin{equation*}
\mathrm{En}(A) \colonequals \underline{A} \in \mathrm{Tm}_{\mathbb{UPG}_\oc}(\Gamma, \mathcal{U}_0);
\end{equation*}

\item If $A$ is $\mathcal{U}_i$ for some natural number $i \in \mathbb{N}$, then
\begin{equation*}
\mathrm{En}(\mathcal{U}_i) \colonequals \underline{\mathcal{U}_i} \in \mathrm{Tm}_{\mathbb{UPG}_\oc}(\Gamma, \mathcal{U}_{i+1});
\end{equation*}

\item If $A$ is $Y(B, C)$, where $Y$ is $\Pi$ or $\Sigma$, then
\begin{equation*}
\mathrm{En}(Y(B, C)) \colonequals q^{\mathrm{OQ}} . \sharp (Y)^{\mathrm{PA}} . \langle \mathrm{En}(B), \lambda \circ \mathrm{En}(C) \rangle \in \mathrm{Tm}_{\mathbb{UPG}_\oc}(\Gamma, \mathcal{U}_{\max (k(B), k(C))});
\end{equation*}

\item If $A$ is $\mathrm{Id}_D(\delta, \delta')$, then
\begin{equation*}
\mathrm{En}(\mathrm{Id}_D(\delta, \delta')) \colonequals q^{\mathrm{OQ}} . \sharp (\mathrm{Id})^{\mathrm{PA}} . \langle \mathrm{En}(D), \langle \delta, \delta' \rangle \rangle \in \mathrm{Tm}_{\mathbb{UPG}_\oc}(\Gamma, \mathcal{U}_{k(D)}).
\end{equation*}

\end{enumerate}

\item \textsc{(U-Elim)} We define the function $\mathrm{El}_k : \mathrm{Tm}_{\mathbb{UPG}_\oc}(\Gamma, \mathcal{U}_k^{[\Gamma]}) \cong \mathbb{WPG}_\oc(\Gamma, \mathcal{U}_k) \rightarrow \mathscr{D}(\Gamma)$ to be the union
\begin{equation*}
\mathrm{El}_k \colonequals \big(\bigcup_{i \in \mathbb{N}} \mathrm{El}_k^{(i)}\big) : \mathbb{WPG}_\oc(\Gamma, \mathcal{U}_k) \rightarrow \mathscr{D}(\Gamma)
\end{equation*}
up to the isomorphism $\mathrm{Tm}_{\mathbb{UPG}_\oc}(\Gamma, \mathcal{U}_k^{[\Gamma]}) \cong \mathbb{WPG}_\oc(\Gamma, \mathcal{U}_k)$, where the function $\mathrm{El}_k^{(i)} : \mathbb{WPG}_\oc(\Gamma, \mathcal{U}_k^{(i)}) \rightarrow \mathscr{D}(\Gamma)$ is given in Definition~\ref{DefUniversePredicativeGames}.
Note that $\mathrm{El}^{(0)}_{k, \Gamma}(\psi) \in \mathscr{D}(\Gamma)$ for each $\psi \in \mathrm{Tm}_{\mathbb{UPG}_\oc}(\Gamma, \mathcal{U}_k^{[\Gamma]})$ is given by
\begin{mathpar}
|\mathrm{El}_{k}(\psi)| \colonequals \bigcup_{\gamma_0^\dagger \in \mathbb{WPG}_\oc(\oc \Gamma)} \mathrm{El}_k(\psi \bullet \gamma_0)
\and
\|\mathrm{El}_{k}(\psi)\| : \gamma_0^\dagger \mapsto \mathrm{El}_k(\psi \bullet \gamma_0).
\end{mathpar}

\item \textsc{(U-Comp)} We see that the equation $\mathrm{El}(\mathrm{En}(A)) = A$ holds by induction on $A$, where we focus on the cases of $A = \Pi(B, C)$ and $A = \mathrm{Id}_D(\delta, \delta')$ since the other cases are similar or trivial. 
\begin{enumerate}

\item Assume $A = \Pi(B, C)$. 
The dependent p-game
\begin{align*}
\mathrm{El} \circ \mathrm{En}(\Pi(B, C)) &= \mathrm{El}(q . \sharp(\Pi) . \langle \mathrm{En}(B), \lambda \circ \mathrm{En}(C) \rangle)
\end{align*}
consists of the underlying p-game
\begin{align*}
|\mathrm{El} \circ \mathrm{En}(\Pi(B, C))| &= |\mathrm{El}(q . \sharp(\Pi) . \langle \mathrm{En}(B), \lambda \circ \mathrm{En}(C) \rangle)| \\
&=  \bigcup_{\gamma_0^\dagger \in \mathbb{UPG}_\oc(\oc \Gamma)} |\Pi(\mathrm{El}(\mathrm{En}(B) \bullet \gamma_0), \mathrm{El}(\lambda \circ \mathrm{En}(C) \bullet \gamma_0))| \\
&=  \bigcup_{\gamma_0^\dagger \in \mathbb{UPG}_\oc(\oc \Gamma)} (|\mathrm{El}(\mathrm{En}(B) \bullet \gamma_0)| \Rightarrow | \mathrm{El}(\lambda \circ \mathrm{En}(C) \bullet \gamma_0)|) \\
&=  \bigcup_{\gamma_0^\dagger \in \mathbb{UPG}_\oc(\oc \Gamma)} |\mathrm{El}(\mathrm{En}(B) \bullet \gamma_0)| \Rightarrow \bigcup_{\gamma_0^\dagger \in \mathbb{UPG}_\oc(\oc \Gamma)}| \mathrm{El}(\lambda \circ \mathrm{En}(C) \bullet \gamma_0)| \\
&= |\mathrm{El} \circ \mathrm{En}(B)| \Rightarrow |\mathrm{El} \circ \mathrm{En}(C)| \\
&= |B| \Rightarrow |C| \quad \text{(by the induction hypothesis)} \\
&= |\Pi(B, C)|
\end{align*}
and the function
\begin{align*}
\| \mathrm{El} \circ \mathrm{En}(\Pi(B, C)) \| : \gamma_0^\dagger \in \mathbb{UPG}_\oc(\oc \Gamma) &\mapsto \mathrm{El}(q . \sharp(\Pi) . \langle \mathrm{En}(B) \bullet \gamma_0, \lambda \circ \mathrm{En}(C) \bullet \gamma_0 \rangle) \\
&= \Pi (\mathrm{El}(\mathrm{En}(B) \bullet \gamma_0), \mathrm{El}(\lambda \circ \mathrm{En}(C) \bullet \gamma_0)) \\
&= \Pi ( \mathrm{El} \circ \mathrm{En}(B)(\gamma_0^\dagger),  \mathrm{El} \circ \mathrm{En}(C)_{\gamma_0^\dagger}) \\
&= \Pi (B(\gamma_0^\dagger), C_{\gamma_0^\dagger}) \quad \text{(by the induction hypothesis)} \\
&= \Pi(B, C)(\gamma_0^\dagger).
\end{align*}
Hence, we have shown the equation
\begin{equation*}
\mathrm{El} \circ \mathrm{En}(\Pi(B, C)) = \Pi(B, C).
\end{equation*}

\item Assume $A = \mathrm{Id}_D(\delta, \delta')$. 
The dependent p-game
\begin{align*}
\mathrm{El} \circ \mathrm{En}(\mathrm{Id}_D(\delta, \delta')) &= \mathrm{El}(q . \sharp(\mathrm{Id}) . \langle \mathrm{En}(D), \langle \delta, \delta' \rangle \rangle)
\end{align*}
consists of the underlying p-game
\begin{align*}
|\mathrm{El} \circ \mathrm{En}(\mathrm{Id}_D(\delta, \delta'))| &= |\mathrm{El}(q . \sharp(\mathrm{Id}) . \langle \mathrm{En}(D), \langle \delta, \delta' \rangle \rangle)| \\
&= \bigcup_{\gamma_0^\dagger \in \mathbb{UPG}_\oc(\oc \Gamma)} |\mathrm{Id}_{\mathrm{El}(\mathrm{En}(D) \bullet \gamma_0)}(\delta \bullet \gamma_0, \delta' \bullet \gamma_0)| \\
&= T' \\
&= |\mathrm{Id}_D(\delta, \delta')|
\end{align*}
and the function
\begin{align*}
\| \mathrm{Id}_D(\delta, \delta') \| : \gamma_0^\dagger \in \mathbb{UPG}_\oc(\oc \Gamma) &\mapsto \mathrm{El}(q . \sharp(\mathrm{Id}) . \langle \mathrm{En}(D) \bullet \gamma_0, \langle \delta \bullet \gamma_0, \delta' \bullet \gamma_0 \rangle \rangle) \\
&= \mathrm{Id}_{\mathrm{El}(\mathrm{En}(D) \bullet \gamma_0)}( \delta \bullet \gamma_0, \delta' \bullet \gamma_0) \\
&= \mathrm{Id}_{\mathrm{El} \circ \mathrm{En}(D)(\gamma_0^\dagger)} (\delta \bullet \gamma_0, \delta' \bullet \gamma_0) \\
&= \mathrm{Id}_{D(\gamma_0^\dagger)} (\delta \bullet \gamma_0, \delta' \bullet \gamma_0) \quad \text{(by the induction hypothesis)} \\
&= \mathrm{Id}_D(\delta, \delta') (\gamma_0^\dagger) \quad \text{(by the equation \ref{IdSubstBeta})}.
\end{align*}
Hence, we have shown the equation
\begin{equation*}
\mathrm{El} \circ \mathrm{En}(\mathrm{Id}_D(\delta, \delta')) = \mathrm{Id}_D(\delta, \delta').
\end{equation*}

\end{enumerate}

\item \textsc{(U-Cumul)} By construction, $\psi \in \mathrm{Tm}_{\mathbb{UPG}_\oc}(\Gamma, \mathcal{U}_k)$ implies $\psi \in \mathrm{Tm}_{\mathbb{UPG}_\oc}(\Gamma, \mathcal{U}_{k+1})$.

\item \textsc{(U-Subst)} By construction, the equation $\mathcal{U}_k^{[\Gamma]}\{ \phi \} = \mathcal{U}_k^{[\Delta]} \in \mathrm{Ty}(\Delta)$ holds.

\item \textsc{($\mathrm{En}$-Subst)} We see that the equation $\mathrm{En}(A)\{ \phi \} = \mathrm{En}(A \{ \phi \}) \in \mathrm{Tm}(\Delta, \mathcal{U})$ holds by induction on $A$, where again we focus on the cases of $A = \Pi(B, C)$ and $A = \mathrm{Id}_D(\delta, \delta')$.
\begin{enumerate}

\item Assume $A = \Pi(B, C)$.
We have the equation
\begin{align*}
\mathrm{En}(\Pi(B, C))\{ \phi \} &= q . \sharp(\Pi) . \langle \mathrm{En}(B) \bullet \phi, \lambda \circ \mathrm{En}(C) \bullet \phi \rangle \\
&= q . \sharp(\Pi) . \langle \mathrm{En}(B\{ \phi \}), \lambda \circ \mathrm{En}(C\{ \phi_B^+ \}) \rangle \quad \text{(by the induction hypothesis)} \\
&= \mathrm{En}(\Pi(B\{ \phi \}, C\{ \phi_B^+ \})) \\
&= \mathrm{En}(\Pi(B, C)\{ \phi \}) \quad \text{(by the equation \ref{PiSubst})}.
\end{align*}

\item Assume $A = \mathrm{Id}_D(\delta, \delta')$.
We have the equation
\begin{align*}
\mathrm{En}(\mathrm{Id}_D(\delta, \delta'))\{ \phi \} &= q . \sharp(\mathrm{Id}) . \langle \mathrm{En}(D) \bullet \phi, \langle \delta \bullet \phi, \delta' \bullet \phi \rangle \rangle \\
&= q . \sharp(\mathrm{Id}) . \langle \mathrm{En}(D\{ \phi \}), \langle \delta\{\phi\}, \delta'\{\phi\} \rangle \rangle \quad \text{(by the induction hypothesis)} \\
&= \mathrm{En}(\mathrm{Id}_{D\{ \phi \}}(\delta\{ \phi \}, \delta'\{ \phi \})) \\
&= \mathrm{En}(\mathrm{Id}_D(\delta, \delta')\{ \phi \}) \quad \text{(by \textsc{Id'-Subst})}.
\end{align*}
\end{enumerate}


\end{itemize}
We have verified all the required axioms, completing the proof.
\end{proof}

\begin{example}
Let us consider the interpretation of the encoding 
\begin{equation*}
\mathsf{f : N \Rightarrow N, g : N \Rightarrow N \vdash En_0(Id_{N \Rightarrow N}(f, g)) : U_0}
\end{equation*}
of the Id-type discussed in \S\ref{HowToEncodeGamesByStrategies}.
The strategy 
\begin{equation*}
\psi \colonequals \mathrm{En}_0(\mathrm{Id}_{N \Rightarrow N}(\pi_1, \pi_2)): (N \Rightarrow N) \mathbin{\&} (N \Rightarrow N) \rightarrow \mathcal{U}_0
\end{equation*}
that interprets this encoding of the Id-type plays as in Figure~\ref{FigEncodingOfId}.
\begin{figure}
\begin{mathpar}
\begin{tabular}{ccccccccccccccc}
$(N$ & $\stackrel{f}{\Rightarrow}$ & $N)$ & $\&$ & $(N$ & $\stackrel{g}{\Rightarrow}$ & $N)$ & $\stackrel{\psi}{\rightarrow}$ & & & & $\mathcal{U}$ \\ \hline
&&&&&&&&&&& $q$ \\
&&&&&&&&&&& $\sharp(\mathrm{Id})$ \\
&&&&&&&&& $q$ \\
&&&&&&&&& $\sharp(\Pi)$ \\
&&&&&&&& $q$ \\
&&&&&&&& $\sharp(N)$
\end{tabular}
\and
\begin{tabular}{ccccccccccccccc}
$(N$ & $\stackrel{f}{\Rightarrow}$ & $N)$ & $\&$ & $(N$ & $\stackrel{g}{\Rightarrow}$ & $N)$ & $\stackrel{\psi}{\rightarrow}$ & & & & $\mathcal{U}$ \\ \hline
&&&&&&&&&&& $q$ \\
&&&&&&&&&&& $\sharp(\mathrm{Id})$ \\
&&&&&&&&& $q$ \\
&&&&&&&&& $\sharp(\Pi)$ \\
&&&&&&&&&& $q$ \\
&&&&&&&&&& $\sharp(N)$
\end{tabular}
\and
\begin{tabular}{cccccccccccc}
$(N$ & $\stackrel{f}{\Rightarrow}$ & $N)$ & $\&$ & $(N$ & $\stackrel{g}{\Rightarrow}$ & $N)$ & $\stackrel{\psi}{\rightarrow}$ & $\mathcal{U}$ && \\ \hline
&&&&&&&& $q$ \\
&&&&&&&& $\sharp(\mathrm{Id})$ \\
&&&&&&&&&&& $q$ \\
&&$q$ \\
$q$ \\
&&&&&&&&& $q$ \\
&&&&&&&&& $n$ \\
$n$ \\
&& $f(n)$ \\
&&&&&&&&&&& $f(n)$
\end{tabular}
\and
\begin{tabular}{ccccccccccccccc}
$(N$ & $\stackrel{f}{\Rightarrow}$ & $N)$ & $\&$ & $(N$ & $\stackrel{g}{\Rightarrow}$ & $N)$ & $\stackrel{\psi}{\rightarrow}$ & $\mathcal{U}$ &&&&& \\ \hline
&&&&&&&& $q$ \\
&&&&&&&& $\sharp(\mathrm{Id})$ \\
&&&&&&&&&&&&&& $q$ \\
&&&&&&$q$ \\
&&&&$q$ \\
&&&&&&&&&&&& $q$ \\
&&&&&&&&&&&& $m$ \\
&&&&$m$ \\
&&&&&& $g(m)$ \\
&&&&&&&&&&&&&& $g(m)$
\end{tabular}
\end{mathpar}
\caption{The strategy on the encoding of the Id-type between functions}
\label{FigEncodingOfId}
\end{figure}
\end{example}

\begin{example}
The elimination rule of N-type with respect to a universe generates the encodings of \emph{transfinite} dependent types.
For instance, the encoding of the type $\mathsf{x : N \vdash List_N(x) \ type}$ of finite lists of natural numbers, which satisfies the judgemental equalities $\mathsf{List_N(\underline{0}) \equiv 1}$ and $\mathsf{List_N(\underline{n+1}) \equiv List_N(\underline{n}) \times N}$, is defined by applying the elimination rule of N-type to the terms
\begin{mathpar}
\mathsf{\vdash En(1) : U}
\and
\mathsf{x : N, y : U \vdash En(El(y) \times N) : U}.
\end{mathpar}

Then, the strategy 
\begin{equation*}
\psi' \colonequals \mathscr{R}^N(\mathrm{En}(1), \mathrm{En}(\mathrm{El}(\pi_2) \mathbin{\&} N)) : N \rightarrow \mathcal{U}_0
\end{equation*}
that interprets this encoding of the list type plays as in Figure~\ref{FigEncodingOfList}.
\begin{figure}
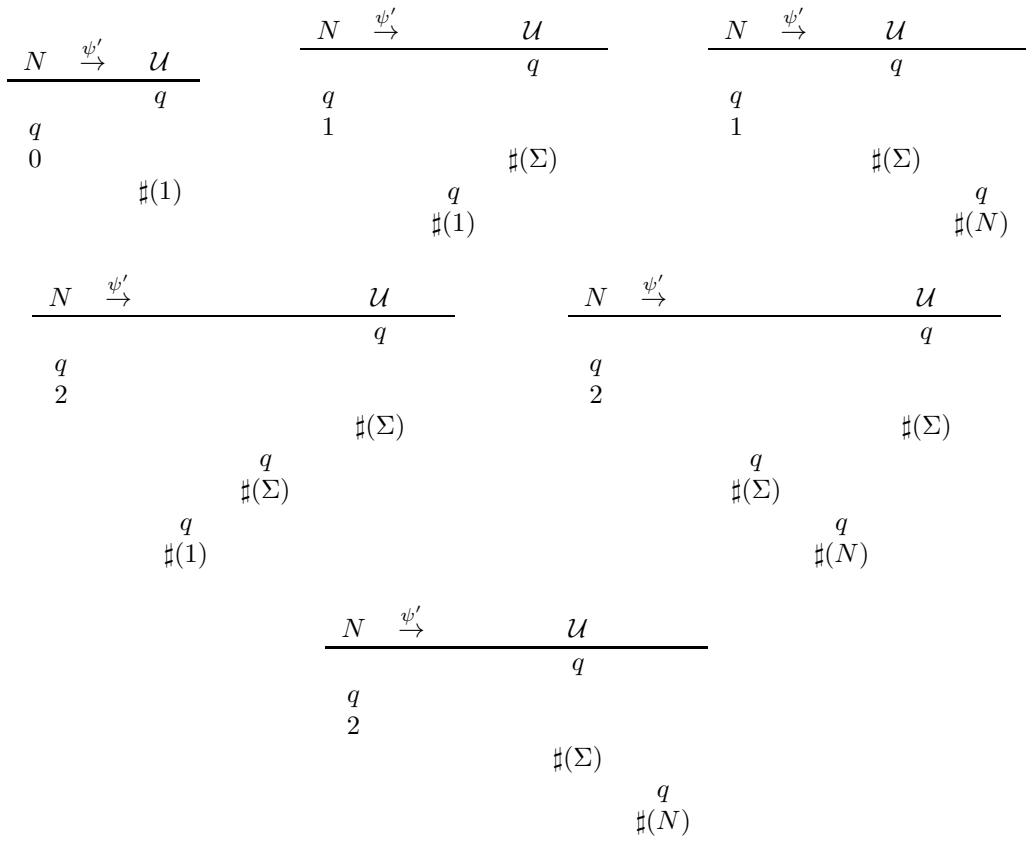

\begin{mathpar}
\begin{tabular}{ccc}
$N$ & $\stackrel{\psi'}{\rightarrow}$ & $\mathcal{U}$ \\ \hline
&&$q$ \\
$q$ \\
$0$ \\
&& $\sharp(1)$
\end{tabular}
\and
\begin{tabular}{ccccc}
$N$ & $\stackrel{\psi'}{\rightarrow}$ & & $\mathcal{U}$ & \\ \hline
&&&$q$ \\
$q$ \\
$1$ \\
&&& $\sharp(\Sigma)$ \\
&& $q$ \\
&& $\sharp(1)$
\end{tabular} 
\and
\begin{tabular}{ccccc}
$N$ & $\stackrel{\psi'}{\rightarrow}$ & & $\mathcal{U}$ & \\ \hline
&&&$q$ \\
$q$ \\
$1$ \\
&&& $\sharp(\Sigma)$ \\
&&&& $q$ \\
&&&& $\sharp(N)$
\end{tabular} 
\and
\begin{tabular}{ccccccc}
$N$ & $\stackrel{\psi'}{\rightarrow}$ &&&& $\mathcal{U}$ & \\ \hline
&&&&&$q$ \\
$q$ \\
$2$ \\
&&&&& $\sharp(\Sigma)$ \\
&&&$q$ \\
&&&$\sharp(\Sigma)$ \\
&&$q$ \\
&&$\sharp(1)$
\end{tabular} 
\and
\begin{tabular}{ccccccc}
$N$ & $\stackrel{\psi'}{\rightarrow}$ &&&& $\mathcal{U}$ & \\ \hline
&&&&&$q$ \\
$q$ \\
$2$ \\
&&&&& $\sharp(\Sigma)$ \\
&&&$q$ \\
&&&$\sharp(\Sigma)$ \\
&&&&$q$ \\
&&&&$\sharp(N)$
\end{tabular} 
\and
\begin{tabular}{ccccccc}
$N$ & $\stackrel{\psi'}{\rightarrow}$ &&&& $\mathcal{U}$ & \\ \hline
&&&&&$q$ \\
$q$ \\
$2$ \\
&&&&& $\sharp(\Sigma)$ \\
&&&&&&$q$ \\
&&&&&&$\sharp(N)$
\end{tabular} 
\end{mathpar}
\caption{The strategy on the encoding of the list type}
\label{FigEncodingOfList}
\end{figure}

Let us note that this list type is out of the scope of the denotational semantics by Abramsky et al. \cite{abramsky2015games,vakar2018game}, let alone its encoding, because their interpretation is limited to \emph{finite inductive types} \cite[Figure~7]{vakar2018game}; also see \cite[\S 4.3]{yamada2022game} on this point. 
This argument in particular implies that their approach cannot interpret the combination of universes and N-type.
\end{example}

\section{Corollaries}
\label{Corollaries}
This last section presents corollaries of Theorem~\ref{ThmGameSemanticsOfUniverses} established in the previous section.
The first corollary is the effectivity of the game semantics of universes (\S\ref{Effectivity}), the second one is the independence of the axiom of equality reflection (\S\ref{IndependenceOfEqualityReflection}), and the last one is the independence of Markov's principle (\S\ref{MP}).

\subsection{Effectivity of game semantics}
\label{Effectivity}
Let us first show the effectivity of our interpretation of universes.
Note that strategies in the CwF $\mathbb{UPG}_\oc$ are the conventional ones (\S\ref{GamesAndStrategies}), which are winning and well-bracketed.
Note also that much more unrestricted strategies that interpret terms in the higher-order functional programming language \textsc{PCF} \cite{scott1993type,plotkin1977lcf} are all effective or \emph{recursive}; see \cite[\S 5]{abramsky2000full} and \cite[\S 5.6]{hyland2000full} for the details.
In essence, terms and morphisms in $\mathbb{UPG}_\oc$ are winning, well-bracketed strategies in the game semantics of \textsc{PCF} that satisfy the additional condition imposed by p-games (Definition~\ref{DefStrategiesOnPredicateGames}). 

The definition of recursive strategies is therefore directly applicable to terms and morphisms in $\mathbb{UPG}_\oc$.
Roughly, assuming that moves in games are encodable by natural numbers, a strategy is \emph{recursive} if its computational steps are all computable (with respect to the encoding of moves by natural numbers) in the standard sense of recursion theory \cite{rogers1967theory}.
We then define: 
\begin{definition}[an effective subCwF $\mathbb{UPG}_\oc^{\mathrm{eff}}$]
Let $\mathbb{UPG}_\oc^{\mathrm{eff}} \hookrightarrow \mathbb{UPG}_\oc$ be the lluf substructural CwF of the CwF $\mathbb{UPG}_\oc$ whose terms and morphisms are all recursive.
\end{definition}

Because strategies in $\mathbb{UPG}_\oc$ that interpret terms in MLTT are much more restricted than those that interpret terms in \textsc{PCF}, it is just straightforward\footnote{Again, the point here is that our strategies are just the conventional ones, so the arguments of the existing methods such as \cite[\S 5]{abramsky2000full} and \cite[\S 5.6]{hyland2000full} are directly applicable.} to verify:
\begin{corollary}[effective game semantics of universes]
The CwF $\mathbb{UPG}_\oc^{\mathrm{eff}}$ is well-defined and supports One-, Zero-, N-, Pi-, Sigma- and Id-types as well as the cumulative hierarchy of universes in the same way as $\mathbb{UPG}_\oc$. 
This in particular establishes effective game semantics of universes.
\end{corollary}

This corollary implies that our game semantics of MLTT equipped with the aforementioned types only employs recursive strategies, i.e., the game semantics is \emph{computational}.
Because universes are \emph{types of types} or \emph{sets of sets}, this computational result is nontrivial.

\subsection{Independence of equality reflection}
\label{IndependenceOfEqualityReflection}
Next, let us show the independence of the axiom of \emph{equality reflection} \cite{palmgren1998universes} from MLTT: Given terms $\psi, \psi' \in \mathrm{Tm}(\Gamma, \mathcal{C})$, if $\mathrm{El}(\psi) = \mathrm{El}(\psi') \in \mathrm{Ty}(\Gamma)$, then $\psi = \psi'$.
Then, a key observation is that, by the \emph{intensional} nature of our game semantics, there can be more than one term that encodes the same type. 
For instance, the term $\mathrm{En}(1) \in \mathrm{Tm}(T.N, \mathcal{U})$ that encodes One-type $1 \in \mathrm{Ty}(T.N)$ plays by
\begin{mathpar}
\begin{tabular}{ccc}
$T.N$ & $\stackrel{\mathrm{En}(1)}{\rightarrow}$ & $\mathcal{U}$ \\ \hline
&& $q$ \\
&& $\sharp(1)$
\end{tabular}
\end{mathpar}
while another term $\psi \in \mathrm{Tm}(T.N, \mathcal{U})$ that plays by
\begin{mathpar}
\begin{tabular}{ccc}
$T.N$ & $\stackrel{\psi}{\rightarrow}$ & $\mathcal{U}$ \\ \hline
&& $q$ \\
$q$ && \\
$n$ && \\
&& $\sharp(1)$
\end{tabular}
\end{mathpar}
for all $n \in \mathbb{N}$ also encodes the same type (n.b., this term is given by the elimination rule of N-type).

This argument together with Theorem~\ref{ThmGameSemanticsOfUniverses} immediately implies:
\begin{corollary}[independence of equality reflection]
The axiom of equality reflection is independent from MLTT equipped with One-, Zero-, N-, Pi-, Sigma- and Id-types as well as universes.
\end{corollary}

We have seen that the intensional nature of strategies plays a crucial role for this corollary, but it is not available for other computational semantics such as domains and realisability \cite{palmgren1993information,streicher2012semantics,blot2018extensional}.

\subsection{Independence of Markov's principle}
\label{MP}
Finally, the previous work \cite[\S 4.7]{yamada2022game} shows that \emph{Markov's principle} \cite{markov1962constructive} is invalid in the game semantics, which implies that the principle is \emph{independent} from MLTT equipped with One-, Zero-, N-, Pi-, Sigma- and Id-types.
Markov's principle is a well-known principle in constructive mathematics, and it depends on the school of constructive mathematics whether the principle is to be regarded as \emph{constructive}.
Roughly, the principle postulates that if it is impossible that there is no natural number $n \in \mathbb{N}$ such that $f(n) = 0$ for a function $f : \mathbb{N} \rightarrow \mathbb{N}$, then there \emph{is} a natural number $n' \in \mathbb{N}$ such that $f(n') = 0$.

The proof of this independence result given in the previous work is also valid for the present game semantics \emph{without any modification}.
This immediately extends the independence result to universes:
\begin{corollary}[independence of Markov's principle from universes]
Markov's principle is independent from MLTT equipped with One-, Zero-, N-, Pi-, Sigma- and Id-types as well as universes.
\end{corollary}

Again, this game-semantic proof \cite{yamada2022game} takes advantages of the \emph{intensional} nature of game semantics, which is not available for other computational semantics of MLTT.

Coquand and Manna \cite{mannaa2017independence} show the independence of Markov's principle from MLTT equipped with a single universe for the first time in the literature.
Their independence proof is \emph{syntactic}, which stands in contrast to our game-semantic proof.
As we have mentioned, their syntactic proof is not automatically extendable to other types, and an extension can be nontrivial. 
In contrast, our game-semantic reasoning is \emph{modular}: A meta-theoretic result on MLTT given by our game semantics is automatically extended to new types as soon as the game semantics is extended to the types. 
This is one of the strong advantages of the game-semantic approach for the study of type theory and constructive mathematics.

\section{Conclusion and future work}
\label{ConclusionAndFutureWork}
We have established computational game semantics of the cumulative hierarchy of universes for the first time in the literature. 
We have also applied this game semantics to the meta-theoretic study of MLTT and shown that equality reflection and Markov's principle are both independent from MLTT equipped with the hierarchy of universes, illustrating advantages of the game-semantic approach. 

For future work, we plan to extend the game semantics further to Martin-L\"{o}f's \emph{well-founded tree (W-) types} \cite{martin1982constructive}.
The resulting game semantics will be a very powerful semantic foundation of constructive mathematics, e.g., it will interpret Aczel's constructive set theory (CZF) \cite{aczel1986type} since CZF is translatable into MLTT equipped with universes and W-types.

\bibliographystyle{amsalpha}
\bibliography{LinearLogic,CategoricalLogic,GamesAndStrategies,HoTT,RecursionTheory}

\end{document}